\documentclass[11pt]{amsart}

\usepackage{amsmath,amssymb,latexsym,soul,cite,mathrsfs}

\usepackage{color,enumitem,graphicx}
\usepackage[colorlinks=true,urlcolor=blue,
citecolor=red,linkcolor=blue,linktocpage,pdfpagelabels,
bookmarksnumbered,bookmarksopen]{hyperref}
\usepackage[english]{babel}

\usepackage[left=2.9cm,right=2.9cm,top=2.8cm,bottom=2.8cm]{geometry}
\usepackage[hyperpageref]{backref}

\usepackage[colorinlistoftodos]{todonotes}
\makeatletter
\providecommand\@dotsep{5}
\def\listtodoname{List of Todos}
\def\listoftodos{\@starttoc{tdo}\listtodoname}
\makeatother

\numberwithin{equation}{section}

\newtheorem{theorem}{Theorem}[section]
\newtheorem{proposition}[theorem]{Proposition}
\newtheorem{lemma}[theorem]{Lemma}
\newtheorem{corollary}[theorem]{Corollary}

\newtheorem{claim}[theorem]{Claim}

\newtheorem{remark}{Remark}
\newtheorem{definition}[theorem]{Definition}

\newcommand\R{\mathbb R}

\begin{document}
	
	\title[Existence of a positive solution for a class Schr\"odinger logarithmic equations]
	{Existence of a positive solution for a class of Schr\"odinger logarithmic equations on exterior domains}

	\author{Claudianor O. Alves}
	\author{Ismael S. da Silva}

	\address[Claudianor O. Alves]{\newline\indent Unidade Acad\^emica de Matem\'atica
		\newline\indent 
		Universidade Federal de Campina Grande,
		\newline\indent
		58429-970, Campina Grande - PB - Brazil}
	\email{\href{mailto:coalves@mat.ufcg.edu.br}{coalves@mat.ufcg.edu.br}}
	
	\address[Ismael S. da Silva]
	{\newline\indent Unidade Acad\^emica de Matem\'atica
		\newline\indent 
		Universidade Federal de Campina Grande,
		\newline\indent
		58429-970, Campina Grande - PB - Brazil}
	\email{\href{ismaels@servidor.uepb.edu.br}{ismaels@servidor.uepb.edu.br}}

	\pretolerance10000
	
	
	\begin{abstract}
		\noindent In this paper we will prove the existence of a positive solution for a class of Schr\"odinger logarithmic equation of the form
		\begin{equation}
		\left\{\begin{aligned}
		-\Delta u &+ u =Q(x)u\log u^2,\;\;\mbox{in}\;\;\Omega,\nonumber \\
		&\mathcal{B}u=0 \,\,\, \mbox{on} \,\,\,  \partial \Omega ,
		\end{aligned}
		\right.
		\end{equation} 
		where $\Omega \subset \mathbb{R}^N$, $N \geq 3$, is an \textit{exterior domain}, i.e., $\Omega^c=\mathbb{R}^N \setminus \Omega$ is a bounded smooth domain where $\mathcal{B}u=u$ or $\mathcal{B}u=\frac{\partial u}{\partial \nu}$. We have used new approach that allows us to apply the usual $C^1$-variational methods to get a nontrivial solutions for these classes of problems.    	
	\end{abstract}

	\subjclass[2019]{Primary:35J15, 35J20; Secondary: 26A27} 
	\keywords{Schr\"odinger logarithmic equation, exterior domain, variational methods}

	\maketitle

	\section{Introduction}
	Several works in the theory of the Partial Differential Equations concern the equation
	$$-\Delta u + V(x)u= f(u),\,\,\,\text{in}\,\,\,\Omega, \leqno{(E_1)}$$
	where $V:\mathbb{R}^N\longrightarrow \mathbb{R}$ and $f:\mathbb{R} \to \mathbb{R}$ are continuous functions satisfying some technical conditions with $\Omega\subset \mathbb{R}^N$ being an open set. One of the main reason for this interest is the fact that $(E_1)$ appears in many important physical problems. For example, the study of the solutions of $(E_1)$ plays an important role in the study of \textit{standing wave solutions} for the well known nonlinear Schr\"odinger equation	
	$$
	i\varepsilon\frac{\partial \psi}{\partial t}=-\varepsilon^2 \Delta \psi +(V(x)+K)\psi-f(\psi), \,\,\,\text{in}\,\,\,\Omega. 
	$$
	Of the mathematical point of view, depending of the conditions on $V$, $f$ or $\Omega$, the existence of solutions for $(E_1)$ may be not a simple question. Indeed, if we intend, for example, to apply variational methods in order to obtain solutions of $(E_1)$ when $\Omega$ is an unbounded domain, we need to deal with the lack of compactness of the embedding below
	$$H_0^1(\Omega)\hookrightarrow L^p(\Omega),\,\,\,p\in (2,2^*).$$
		
In 1987, Benci and Cerami \cite{Benci-Cerami} studied the existence of positive solution for the following class of elliptic problems 
	$$\left\{\begin{aligned}
	-&\Delta u + \lambda u= |u|^{p-2}u,\,\,\,\text{in}\,\,\,\Omega,\\
	&u\in H_0^1(\Omega),
	\end{aligned}
	\right.
	\leqno{(S_1)}
	$$
	where $\Omega \subset \mathbb{R}^N$, $N \geq 3$, is an \textit{exterior domain}, i.e., $\Omega^c=\mathbb{R}^N \setminus \Omega$ is a bounded smooth domain. In that paper, the authors showed that $(S_1)$ has at least one positive solution when $\mathbb{R}^N \setminus \Omega  \subset B_\rho(0)$, for $\rho\approx 0^+$ or $\lambda \approx 0^+$. In the study developed in \cite{Benci-Cerami}, it was proved that the problem $(S_1)$ has no \textit{ground state solution}, however by establishing a carefully study of the Palais-Smale sequences (shortly $(PS)$ sequences) of the energy functional associated with $(S_1)$, it was showed that the $(PS)_c$ condition holds for some level (see \cite[Section 3]{Benci-Cerami}), which permitted to establish the existence of solution with high energy. 
	
	Posteriorly, many authors studied the existence of solution for other classes of elliptic problems in exterior domains. Here, we cite, e.g., the works of Bahri and Lions \cite{Bahri-Lions} and Li and Zheng \cite{Li-Zheng}. In those papers, the authors showed the existence of solution for elliptic problems of the type  
	$$\left\{\begin{aligned}
	-&\Delta u + \lambda u= A(x,u),\,\,\,\text{in}\,\,\,\Omega,\\
	&u\in H_0^1(\Omega),
	\end{aligned}
	\right.
	\leqno{(S_2)}
	$$
where $\Omega$ is an exterior domain. In the aforementioned papers, $A$ is a continuous function with subcritical growth satisfying
	$$\lim_{|x| \rightarrow \infty}A(x,t)= A_\infty(t),\,\,\,\forall t \in \mathbb{R}.$$
	It is very important to mention that the uniqueness, up to translations, of positive solutions for the limit problem associated with $(S_2)$, that is,
	$$\left\{\begin{aligned}
	-&\Delta u + \lambda u= A_\infty(u),\,\,\,\text{in}\,\,\,\mathbb{R}^N\\
	&u\in H^1(\mathbb{R}^N),
	\end{aligned}
	\right.
	\leqno{(S_3)}
	$$
	plays an important role in order to find solutions of $(S_2)$.  However, in Alves and de Freitas \cite{Alves- de Freitas}, the authors showed the existence of a positive solution for the following class of problem in an exterior domain
	$$
	\left\{\begin{aligned}
	-&\Delta u + u= u^q +\varepsilon u^{2^*-1},\,\,\,\text{in}\,\,\,\Omega,\\
	&u>0,\,\,\,\text{in}\,\,\,\Omega,\\
	&u\in H_0^1(\Omega).
	\end{aligned}
	\right.
	\leqno{(S_4)}
	$$
	In that case, the limit problem given by 
		$$
	\left\{\begin{aligned}
		-&\Delta u + u= u^q +\varepsilon u^{2^*-1},\,\,\,\text{in}\,\,\, \mathbb{R}^N,\\
		&u>0,\,\,\,\text{in}\,\,\, \mathbb{R}^N,\\
		&u\in H^1(\mathbb{R}^N),
	\end{aligned}
	\right.
	\leqno{(S_5)}
	$$
	does not have uniqueness, up to translations, which brings a lot of difficulties to establish the existence of solution for $(S_4)$. However, by doing some estimates related with the energy functional associated to the problem $(S_4)$, as $\varepsilon$ goes to $0$, the authors were able to adapt some ideas in \cite{Benci-Cerami} and they showed the existence of positive solution for $(S_4)$, with $\varepsilon \approx 0^+$. We also refer to the reader the works \cite{Alves-Carriao-Medeiros, Alves-Ambrosio-Torres, Alves-Bisci-Torres, Alves-Torres, Bernardini-Cesaroni,Cao, Esteban} as an interesting complementary bibliography about problems in exterior domains. In these works, the reader will find varied results on the problem $(S_2)$ involving different boundary conditions or driven by quasilinear operators. 
	
	Elliptic problem in exterior domains were also considered with Neumann boundary condition. In this line, the first paper that we would like to cite is due to Esteban \cite{Esteban}, where it was  proved that the problem $(S_1)$ with the Neumann condition has a ground state solution. In \cite{Cao}, Cao considered the following class of problem in exterior domains 
	$$\left\{\begin{aligned}
	-&\Delta u + u= Q(x)|u|^{p-2}u,\,\,\,\text{in}\,\,\,\Omega,\\
	&\frac {\partial u}{\partial \eta}=0, \,\,\,\text{in}\,\,\, \partial \Omega,
	\end{aligned}
	\right.
	\leqno{(S_6)}
	$$
	with $Q\in C(\mathbb{R}^N,\mathbb{R})$ satisfying
	\begin{center}
		 $Q(x)\geq C_0-Ce^{-\sigma|x|}|x|^{-r}$, $Q(x) \rightarrow C_0>0$, as $|x|\rightarrow \infty$,
	\end{center}
	for some $\sigma$, $s>N-1$ and $C>0$. In that paper, it was proved that the problem $(S_6)$ has a ground state solution. Moreover, by considering an additional condition on $Q$, the author was also able to show the existence of nodal solution for $(S_6)$. In \cite{Alves-Carriao-Medeiros}, Alves, Carri\~ao and Medeiros generalized the results found in \cite{Cao} for the $p$-Laplacian operator. 
	
 Later, in 2020, Alves and Torres Ledesma  \cite{Alves-Torres} studied the problem
	$$\left\{\begin{aligned}
	(-&\Delta)^s u + u= Q(x)|u|^{p-2}u,\,\,\,\text{in}\,\,\,\Omega,\\
	& \mathcal{N}_s u=0, \,\,\,\text{in}\,\,\, \Omega^c,
	\end{aligned}
	\right.
	\leqno{(S_7)}
	$$
	where $\Omega$ is an exterior domain, $(-\Delta)^s$ is the fractional Laplacian operator and $\mathcal{N}_s$is the nonlocal operator related with the Neumann condition. The authors established the existence of ground state solution for $(S_7)$ when $Q$ is a continuous function satisfying 
	\begin{center}
	$Q(x)\geq C_0>0$, $Q(x) \rightarrow C_0$, as $|x|\rightarrow \infty$.
	\end{center} 
	The existence of nodal solution was also investigated in that paper. 
	
	Still related with the aforementioned equation $(E_1)$, several recent works have studied the existence of solution for the equation $(E_1)$  for the case $f(t)=t\log t^2$, that is,
	$$ -\Delta u + V(x)u = u\log u^2,\,\,\,\text{in}\,\,\,\mathcal{O}, \leqno{(E_2)}$$
	where $\mathcal{O} \subset \mathbb{R}^N$ is a domain not necessarily bounded. The above equation has received a special attention, because it is related with relevant physical questions, such as   quantum mechanics, quantum optics, effective quantum gravity, Bose-Einstein condensation and others interesting physical topics , see \cite{Zloshchastiev} for more details. 
	
	Besides of the importance for the mentioned physical subjects, the existence of solution for $(E_2)$ is not simple when $\mathcal{O}$ is unbounded,  because it involves nontrivial questions in the mathematical point of view. For example, if we intend to apply variational methods to get solutions of $(E_2)$, the standard possibility for the energy functional associated to $(E_2)$ is the functional given by
	$$J(u)=\frac{1}{2}\int_{\mathcal{O}}(|\nabla u|^2+(V(\varepsilon x))|u|^2+1)dx-\frac{1}{2}\int_{\mathcal{O}}u^2 \log |u|^2dx,\,\,\,u \in H_0^1(\mathcal{O}) \,(\text{or}\,\,\,u\in H^1(\mathcal{O})),$$ 
	(We have used that $\displaystyle{\int_{0}^{t}s\log s^2 \,ds}=\frac{1}{2}t^2 \log t^2-\frac{t^2}{2}$). However, when $\mathcal{O}$ is an unbounded domain, we cannot ensure that  $J(u) \in \mathbb{R}$  for all $ u \in H^1(\mathcal{O})$. For example, if $\mathcal{O}=\mathbb{R}^N$,  we take $u_*$ a smooth function satisfying
	$$u_*(x)= \left\{\begin{aligned}
	&(|x|^{N/2}\log |x|)^{-1},\,\,\,&|x|\geq 3;\\
	&0,\,\,\,&|x|\leq 2,
	\end{aligned}
	\right.    
	$$
which satisfies $u_*\in H^1(\mathbb{R}^N)$, but $\displaystyle{\int_{\mathbb{R}^N}}|u_*|^2 \log |u_*|^2dx= -\infty$, and so, $J(u_*)=+\infty$. 
	
Some different approaches have been explored in a lot of recent papers to overcome these difficulties inherent to the study of the equations of $(E_2)$-type. In \cite{Cazenave},  Cazenave studied a logarithmic Schr\"odinger equation, namely
	$$\leqno{(E_3)}\hspace{4 cm}iu_t+\Delta u + u\log u^2=0,\,\,\,\,x \in \mathbb{R}\times\mathbb{R}^N,$$
	by working on the space 
	$$
	W:=\left\{u \in H^1(\mathbb{R}^N);\,\displaystyle{\int_{\mathbb{R}^N}}\mid u^2\log u^2\mid dx<\infty\right\}
	$$ 
	and considering a different topology of the usual topology of $H^1(\mathbb{R}^N)$. In that paper, the author provided a lot of information about the solutions of $(E_3)$ that are the form $\varphi(t,x):=e^{-iHt}\phi(x)$, where $H$ is a constant and $\phi$ verifies the problem
	$$
	-\Delta \phi -H\phi=\phi \log \phi^2 \quad \mbox{in} \quad \mathbb{R}^N.
	$$  
	
	Later on, several works have considered the problem of the existence of solution for the following class of problems
	$$\left\{\begin{aligned}
	-\epsilon^2 \Delta u &+ V(x)u  =u\log u^2,\;\;\mbox{in}\;\;\mathbb{R}^{N},\nonumber \\
	&u \in H^{1}(\mathbb{R}^{N}),
	\end{aligned}
	\right.\leqno{(S_5)}
	$$
where $\epsilon>0$ and $V: \mathbb{R}^N \to \mathbb{R}$ is a potential satisfying some technical conditions, see for example \cite{Alves-de Morais, Alves-Ji, Alves-Ji2, Alves-Ji3, d'Avenia, Ji-Szulkin, Squassina-Szulkin} and references therein. In these references, we find some results about the existence, concentration and multiplicity of solutions of $(S_5)$. In general, in the above mentioned papers, the authors considered a decomposition of the logarithmic term of the form
	\begin{equation}\label{052}
	F_2(t)-F_1(t)=\frac{1}{2}t^2 \log t^2\,\,\,\, \forall t \in \mathbb{R},
	\end{equation}
	where $F_1$, $F_2 \in C^1(\mathbb{R})$, $F_1$ is a convex function, and $F_2$ is a function with a subcritical growth (see Subsection 3 below). This little trick allow us to decompose the functional $J$ above as a sum of a $C^1$-functional with a lower semicontinuous and convex functional, which is a key point to apply the nonsmooth Critical Point Theory introduced by Szulkin \cite{Szulkin}. However, due to the lack of smoothness, some questions related with the existence of critical point on $C^1$-manifolds cannot be attacked with this approach. For example, we could not reproduce directly the reasoning presented in \cite{Benci-Cerami} to study the existence of solution for a problem in exterior domain involving logarithmic nonlinearities, because the approach used in \cite{Benci-Cerami} deals with the notion of critical points on the \textit{Nehari set} that is a $C^1$-manifold.
	
	Motivated by the aforementioned questions related with the logarithmic equations and elliptic problems in exterior domains, in the present paper our main goal is to study the following class of problems
	\begin{equation}
	\left\{\begin{aligned}
	-\Delta u &+ u =Q(x)u\log u^2,\;\;\mbox{in}\;\;\Omega,\nonumber \\
	&\mathcal{B}u=0 \,\,\, \mbox{on} \,\,\,  \partial \Omega,
	\end{aligned}
	\right.
	\end{equation}
where $\Omega\subset \mathbb{R}^N$ ($N\geq 3$) is an exterior domain, $Q\in C(\mathbb{R}^N,\mathbb{R})$ and $\mathcal{B}$ satisfying $\mathcal{B}u=u$ or $\mathcal{B}u=\frac{\partial u}{\partial \nu}$. 
Firstly, we consider $Q\equiv 1$ with the \textit{Dirichlet condition} on $\partial \Omega$, i.e., $\mathcal{B}u=u$. For this case, our problem can be rewritten as
	$$
	\left\{\begin{aligned}
	-\Delta u &+ u =u\log u^2,\;\;\mbox{in}\;\;\Omega,\nonumber \\
	&u \in H_0^{1}(\Omega). 
	\end{aligned}
	\right.\leqno{(P_0)} 
	$$

	In order to prove our results, we will follow the approach introduced by Alves and da Silva \cite{Alves-S da Silva} that permits to adapt some ideas explored in \cite{Alves- de Freitas, Benci-Cerami}.  In \cite{Alves-S da Silva},   the authors proved that the function $F_1$ in the decomposition given in (\ref{052}) is a $N$-function, and then they have used this information to introduce a new reflexive and separable Banach space on which the energy functional associated with $(P_0)$ is a $C^1$-functional.
	
	The main result associated with $(P_0)$ is the following: 
	\begin{theorem}\label{053}
		There exists $\rho_0\approx 0^+$ such that, if $\Omega^c \subset B_\rho(0)$, then the problem $(P_0)$ has a positive solution for each $\rho \in (0,\rho_0)$.
	\end{theorem}

	In our second main result, we investigate the existence of ground state solution when \linebreak $\mathcal{B}u:=\dfrac{\partial u}{\partial \eta}$ (\textit{Neumann condition}), more precisely, we will consider the following class of problem  
	$$
	\left\{\begin{aligned}
		-&\Delta u + u= Q(x)u\log u^2,\,\,\,\text{in}\,\,\,\Omega\\
		&\frac{\partial u}{\partial \eta}=0,\,\,\,\text{on}\,\,\,\partial \Omega,
	\end{aligned}
	\right.
	\leqno{(S_0)}
	$$
where $\Omega$ is an exterior domain. On the function $Q$, we assume the following conditions:
	\begin{itemize}
		\item [$(Q_1)$] $\displaystyle{\lim_{|x|\rightarrow \infty}} Q(x) = Q_0$ and  $q_0:=\displaystyle{\inf_{x\in \mathbb{R}^N}}\,Q(x)>0$ for all $x\in \mathbb{R}^N$;\\
		\item [$(Q_2)$] $Q_0\geq Q(x)\geq Q_0-Ce^{-M_0|x|^2}$, for $x\geq R_0$,\\
	\end{itemize}
	with $Q_0, C, M_0$, $R_0>0$.

	Our main second result is the following 
	\begin{theorem}\label{068}
		If the conditions $(Q_1)-(Q_2)$ hold, then for some $M_0$ large enough, the problem $(S_0)$ has a positive ground state solution.
	\end{theorem}
	
	We would like to emphasize that associated with problem $(P_0)$, in the same line of \cite{Alves- de Freitas,Benci-Cerami}, we have proved a result of nonexistence of ground state solution for the problem $(P_0)$, for more details see Theorem \ref{034}.
	
	Since the main results above ensure the existence of a positive solution for the problems $(P_0)$ and $(S_0)$ respectively, the present paper complements the study made for the equation $(E_1)$, because we are studying the existence of solution for a class of elliptic problems in exterior domains with logarithmic nonlinearity. 
	
	Before finishing this section, we would like to point out some contributions that were obtained as byproduct of our results:
	\begin{enumerate}
		\item [$i)$:] By using the approach introduced in \cite{Alves-S da Silva}, we have showed that the Nehari set $\mathcal{N}$ related with the energy functional $I$ associated with $(P_0)$ is a $C^1$-manifold, which permits to prove the existence of critical points for $I|_{\mathcal{N}}$. Here, it is proved that these critical points are in fact critical points of $I$ in the usual sense (see Proposition \ref{07} );
		\item [$ii)$:] In the study of problem $(P_0)$, we established a compactness lemma that gives the same type of information of \cite[Lemma 3.1]{Benci-Cerami} (see the Lemma \ref{031} ). Our result allowed us to do a precise description of the behavior of $(PS)$ sequences of $I|_{\mathcal{N}}$. 
		\item [$iii)$:] In the second problem, more precisely the problem $(S_0)$, we have showed that a similar compactness lemma holds for the Neumann case by working with the \textit{Cerami sequences} (see the Lemma \ref{067} for more details.)
	\end{enumerate}
	
	The paper is organized as follows: In Section 2, we present a brief review about $N$-functions and Orlicz spaces. In Section 3, we recall the main tools associated with the approach explored in \cite{Alves-S da Silva}, it is proved the nonexistence of ground state solution for $(P_0)$  as well as a compactness lemma for the functional $I|_{\mathcal{N}}$. In Section 4, the main technical results used to obtain the existence of solution of $(P_0)$ are established. Finally, in Section 5-6, we prove the Theorems \ref{053} and \ref{068}. 
	\\ 
	\\
	\noindent \textbf{Notation:} 
	\begin{enumerate}
		\item[$\bullet$] $\|\cdot\|_p$ denotes the usual norm of the Lebesgue space $L^p(\mathbb{R}^N)$, $p \in [1, \infty];$
		\item[$\bullet$] If $f:A \rightarrow \mathbb{R}$ is a measurable function, with $A \subset \mathbb{R}^N$ being a measurable set, then $\displaystyle{\int_{A}} f(x)\,dx$ will be denoted by $\displaystyle{\int_{A}}f\,;$
		\item[$\bullet$] $o_n(1)$ denotes a real sequence with $o_n(1) \rightarrow 0;$
		\item[$\bullet$] $C_{x_1,...,x_n}$ denotes a positive constant which depends of $x_1,...,x_n;$
		\item[$\bullet$] $2^*:= \displaystyle\frac{2N}{N-2}$, if $N\geq 3$ and $2^*:=\infty$ if $N=1$ or $N=2$.	
	\end{enumerate}

	\section{A short review about Orlicz spaces}
	
	In this section, we present some notions and properties related to the Orlicz spaces, for more details see \cite{Adams1, Fukagai 1, RAO}. 
	
	\begin{definition}\label{N}
		A continuous function $\Phi: \mathbb{R}\rightarrow [0, +\infty)$ is a $N$-function if:
		
		\item [(i)] $\Phi$ is convex.
		\item [(ii)] $\Phi(t)=0\Leftrightarrow t=0$.
		\item [(iii)] $\displaystyle \lim_{t\rightarrow 0} \frac{\Phi(t)}{t}=0$ and  $\displaystyle\lim_{t\rightarrow \infty} \frac{\Phi(t)}{t}=+\infty$.
		\item [(iv)] $\Phi$ is an even function.
		
	\end{definition}
	We say that a $N$-function $\Phi$ verifies the $\Delta_{2}$-condition, denoted by $\Phi \in (\Delta_{2})$, if
	\begin{equation*}
	\Phi(2t) \leq k\Phi(t)\;\;\forall\; t\geq t_0,
	\end{equation*}
	for some constants $k>0$ and $t_0 \geq 0$. 
	
	The conjugate function $\tilde{\Phi}$ associated with $\Phi$ is given by the Legendre's transformation, more precisely,
	$$\tilde{\Phi}= \max_{t \geq 0} \{st-\Phi(t)\}\;\;\mbox{for}\;\; s \geq 0.$$
	It is possible to prove that $\tilde{\Phi}$ is also a N-function. The functions $\Phi$ and $\tilde{\Phi}$ are complementary to each other, that is, $\tilde{\tilde{\Phi}}=\Phi$.
	
	Given an open set $A \subset \mathbb{R}^N$, we define the Orlicz space associated with the $N$-function $\Phi$ as
	$$L^{\Phi}(A) = \left\{u \in L^{1}_{loc}(A) \; ; \; \int_{A} \Phi\left(\dfrac{|u|}{\lambda}\right) < +\infty, \quad \text{for some} \, \, \lambda >0 \right\}.$$
	The space $L^{\Phi}(A)$ is a Banach space endowed with Luxemburg norm given by
	$$
	||u||_{\Phi} = \inf \left\{\lambda > 0 \; ; \; \int_{A}\Phi\left(\dfrac{|u|}{\lambda}\right) \le 1 \right\}.
	$$
	
	We would like to point out that in Orlicz spaces we also have a H\"older and Young type inequalities, namely
	$$st\leq \Phi(t)+\tilde{\Phi}(s),\,\,\,\forall s,t\geq 0,$$
	and
	$$
	\left\lvert\int_{A} u v \right\lvert \le 2||u||_{\Phi}||v||_{\tilde{\Phi}},\;\;\forall\; u  \in L^{\Phi}(A) \quad \mbox{and} \quad u  \in L^{\tilde{\Phi}}(A).
	$$
	Moreover, for each $\varepsilon>0$, it holds
	\begin{equation}\label{DY}	
	st\leq \Phi(C_\varepsilon t)+\varepsilon\tilde{\Phi}(s),\,\,\,\forall s,t\geq 0,
	\end{equation}
	for some positive $C_\varepsilon>0$.	
	When $\Phi$, $\tilde{\Phi}\in (\Delta_2)$,  the space $L^{\Phi}(A)$ is reflexive and separable. Furthermore, the $\Delta_{2}$-condition yields that
	$$
	L^{\Phi}(A)=\left\{u \in L^{1}_{loc}(A) \; ; \; \int_{A} \Phi\left(|u|\right) < +\infty \right\}
	$$
	and
	$$
	u_{n}\rightarrow u \;\;\mbox{in}\;\; L^{\Phi}(A)\Leftrightarrow \int_{A} \Phi(|u_{n}-u|)\rightarrow 0.
	$$
	
	We would like to mention an important relation involving $N$- functions, which will be used later on. Let $\Phi$ be a $N$-function of $C^1$ class and $\tilde{\Phi}$ its conjugate function. Suppose that
	
	\begin{equation}\label{FIN}
	1<l\leq\frac{\Phi'(t)t}{\Phi(t)}\leq m<N,\,\,\,\, t\neq 0,
	\end{equation}
	then $\Phi$, $\tilde{\Phi} \in (\Delta_2)$. 
	
	Finally, setting the functions 
	$$\xi_0(t):=\min\{t^l,\,t^m\}\,\,\, \text{and}\,\,\, \xi_1(t):\max\{t^l,\,t^m\},\,\,\,\,\quad t\geq0,$$
	it is well known that under the condition (\ref{FIN}) one has 
	\begin{equation}\label{In}
	\xi_0(||u||_{\Phi})\leq \int_{A}\Phi(u) \leq \xi_1(||u||_\Phi).
	\end{equation}
	We finish this section by recalling  a Brezis-Lieb type result involving $N$-functions found in \cite[Theorem 2]{BrezisLieb}
	\begin{proposition}[A Brezis-Lieb type result]\label{BL}
		Suppose $\Phi$ is a $N$-function with $\Phi \in (\Delta_2)$. Let $(g_n)$ be a sequence in $L^{\Phi}(A)$ satisfying:
		\begin{itemize}
			\item[$i)$] $(g_n)$ is a bounded sequence in $L^{\Phi}(\Omega)$;
			\item[$ii)$] $g_n(x)\rightarrow 0$ a.e. in $A$. 	
		\end{itemize}	
		Then, for each $w \in L^{\Phi}(A)$, 
		$$\int_{A}|\Phi(g_n+w)-\Phi(g_n)-\Phi(w)|=o_n(1).$$
	\end{proposition}

	\section{The variational framework}
	This section is devoted to show some technical results that will be used later on. We start by recalling an important result involving the uniqueness of solution for the logarithmic equation on the whole $\mathbb{R}^N$. After that, we present the function space that will consider and prove a result of nonexistence of ground state solution for $(P_0)$,  which is an important step in our study. Finally, we prove a compactness lemma that is crucial to understand the behavior  of the $(PS)$ sequences of the energy functional associated with $(P_0)$. 
	
	Our first result in this section can be found in \cite[Section 1]{d'Avenia} (see also \cite{Bialynicki-Birula}) and it concerns with the uniqueness of solution for the following class of problem
	\begin{equation} \label{problemalimite}
	\left\{\begin{aligned}
	-\Delta u &+ \kappa u =u\log u^2,\;\;\mbox{in}\;\;\mathbb{R}^N, \\
	&u \in H^{1}(\mathbb{R}^N),
	\end{aligned}
	\right.
	\end{equation}
	where $\kappa>0$. 
	\begin{theorem}\label{1}
		The problem \eqref{problemalimite} has a unique positive solution $u \in C^2(\mathbb{R},\mathbb{R})$, up to translations, such that $u(x)\rightarrow 0$ as $|x|\rightarrow \infty$. More precisely, the solution $u$ is given by
		$$u(x)=C_{\kappa,N}\,e^{\frac{-|x|^2}{2}}.$$
	\end{theorem}	
The theorem above ensures that any positive solution of \eqref{problemalimite} has an exponential decaying.

	\subsection{The energy functional}
	In the same way of \cite{Alves-de Morais, Alves-Ji, Alves-S da Silva, Ji-Szulkin}, we will explore  a suitable decomposition of the function
	$$F(s)= \displaystyle{\int_{0}^{s}t\log t^2 \,dt}= \frac{1}{2}s^2 \log s^2-\frac{s^2}{2}, \quad s \in \mathbb{R},$$ 
	which allow us to introduce an energy functional associated with $(P_0)$.
	For each $\delta>0$ sufficiently small, we set
	$$
	F_1(s):=\left\{\begin{aligned}
	&0,  \quad \,& s=0,\\
	-\frac{1}{2} &s^2 \log s^2,\quad &0<|s|<\delta,\\
	-\frac{1}{2} &s^2 (\log \delta^2 +3) + 2\delta|s| - \frac{\delta^2}{2},&|s|\geq \delta,
	\end{aligned}
	\right.
	$$
	and
	$$
	F_2(s):=	\left\{\begin{aligned}
	&0,  \quad \,& s=0\\
	-\frac{1}{2} &s^2 \log \left(\frac{s^2}{\delta^2}\right) + 2\delta|s| -\frac{3}{2}s^2-\frac{\delta^2}{2},\,&|s|\geq \delta,
	\end{aligned}
	\right.
	$$	
	for every $s\in \R.$  A direct computation shows that 
		\begin{equation}\label{3}
	F_2(s)-F_1(s)=\frac{1}{2}s^2 \log s^2,\,\,\,\, \forall s \in \mathbb{R}.
	\end{equation}
	Moreover, it is well known that $F_1$ and $F_2$ satisfy the properties $(P_1)-(P_2)$ below:
	\\ 
	\begin{itemize}
		\item[($P_1)$] $F_1$ is an even function with $F_1'(s)s\geq 0$ and $F_1(s)\geq 0$ for all $s \in \mathbb{R}$. Moreover $F_1 \in C^1(\mathbb{R},\mathbb{R})$ and it is also convex if $\delta \approx 0^+$;
		\item[($P_2)$] $F_2 \in C^1(\mathbb{R},\mathbb{R})$ and for each $p \in (2,2^*)$, there exists $C=C_p>0$ such that
		$$ |F_2'(s)|\leq C|s|^{p-1}\,\,\,\, \forall s \in \mathbb{R}.$$
	\end{itemize}
	
	The below proposition is an important result involving the function $F_1$
	\begin{proposition}\label{4}
		The function $F_1$ is a $N$-function. Furthermore, it holds $F_1$, $\tilde{F_1} \in (\Delta_2)$ and there is $l \in (1,2)$ such that
		\begin{equation} \label{2}
		1<l\leq\frac{F'_1(s)s}{F_1(s)}\leq 2,\,\,\,\, \forall s>0.
		\end{equation} 
	\end{proposition}
	\begin{proof}
		See \cite[Proposition 3.1]{Alves-S da Silva}.
	\end{proof}
	
	By using the definition of the functions $F_1$ and $F_2$ and a Brezis-Lieb type result, it is possible to prove the lemma below whose the idea for the proof can be found in \cite[Lemma~3.1]{Shu}. 
\begin{lemma}\label{Brezis-Lieb}
		Let $(u_n)$ be a bounded sequence in $H^{1}(\mathbb{R}^N)$ such that $u_n \rightarrow u$ a.e. in $\mathbb{R}^N$ and $\{u_n^2 \log u_n^2\}$ is a bounded sequence in $L^1\left(\mathbb{R}^N\right)$. Then,
		$$
		 \int_{\Omega}\left|u_n-u\right|^2 \log \left|u_n-u\right|^2=\int_{\Omega}u_n^2 \log u_n^2-\int_{\Omega}u^2 \log u^2 +o_n(1).
		$$
	\end{lemma}
	
	From now on, we will set $X:=H_0^1(\Omega)\cap L^{F_1}(\Omega)$ endowed with the norm
	$$||\cdot||_X:=||\cdot||_{H_0^1(\Omega)}+||\cdot||_{F_1}.$$
	Here, $L^{F_1}(\Omega)$ designates the Orlicz space associated with $F_1$ and $||\cdot||_{F_1}$ denotes the usual norm associated with $L^{F_1}(\Omega)$.  In view of the last proposition, the space $X$ is a separable and reflexive Banach space. Furthermore, the embeddings  $X\hookrightarrow H^1(\Omega)$ and $X\hookrightarrow L^{F_1}(\Omega)$ are continuous.
	
	The natural candidate for the energy functional associated with $(P_0)$ is given by
	$$I(u):=\frac{1}{2}\int_\Omega(|\nabla u|^2+2|u|^2)+ \int_\Omega F_1(u)-\int_\Omega F_2(u),\,\,\,\, \forall u \in X.$$
	It will be convenient to take the norm of $H_0^1(\Omega)$ as being
	$$||u||_{H_0^1(\Omega)}:=\left(\int_\Omega(|\nabla u|^2+2|u|^2)\right)^{\frac{1}{2}},$$
	which is equivalent to the usual norm of $H_0^1(\Omega)$. Moreover, it is associated with the inner product
	$$
	\langle u,v \rangle_{H_0^1(\Omega)} := \int_{\Omega}(\nabla u\nabla v+2uv),\,\,\,\forall u,v \in H_0^1(\Omega).
	$$
	Similarly, we will consider
	$$
	||u||_{H^1(\mathbb{R}^N)}:=\left(\int_{\mathbb{R}^N}(|\nabla u|^2+2|u|^2)\right)^{\frac{1}{2}},\,\,\,\forall u \in H^1(\mathbb{R}^N),
	$$
	as the norm in $H^1(\mathbb{R}^N)$.
	
	From $(P_1)-(P_2)$, $I \in C^1(X, \mathbb{R})$ and 
	$$I'(u)v=\int_{\Omega}(\nabla u\nabla v+2uv)+\int_{\Omega}F'_1(u)v-\int_{\Omega}F'_2(u)v,\,\,\,\, \forall v \in X.$$
	
In our approach, we will use some properties of the limit problem below
	\begin{equation}
	\left\{\begin{aligned}
	-\Delta u &+ u =u\log u^2,\;\;\mbox{in}\;\;\mathbb{R}^N,\nonumber \\
	&u \in H^{1}(\mathbb{R}^N).
	\end{aligned}
	\right.\leqno{(P_\infty)}
	\end{equation}
	Associated with $(P_\infty)$, we have the functional
	$$I_\infty(u):=\frac{1}{2}\int_{\mathbb{R}^N}(|\nabla u|^2+2u^2)+ \int_{\mathbb{R}^N} F_1(u)-\int_{\mathbb{R}^N} F_2(u), \,\,\,\, \forall u \in Y,$$
	where $Y:=(H^1(\mathbb{R}^N)\cap L^{F_1}(\mathbb{R}^N),||\cdot||_Y)$ and $||\cdot||_Y:=||\cdot||_{H^1(\mathbb{R}^N)}+||\cdot||_{L^{F_1}(\mathbb{R}^N)}$. Related to the functionals $I$ and $I_\infty$, we also have the Nehari sets
	$$\mathcal{N}:=\{u \in X-\{0\};\,I'(u)u=0\}$$
	and
	$$\mathcal{N}_\infty:=\{u \in Y-\{0\};\,I_\infty'(u)u=0\},$$
which can be characterized by 
$$\mathcal{N}:=\Psi_0^{-1}(0)\,\,\,\text{and}\,\,\,\mathcal{N}_\infty:=\Psi_\infty^{-1}(0),$$
with 
\begin{equation}\label{08}
	\Psi_0(u)=I(u)-\frac{1}{2}\int_{\Omega}|u|^2\,\,\,\text{and}\,\,\, \Psi_\infty(u)=I_\infty(u)-\frac{1}{2}\int_{\mathbb{R}^N}|u|^2.
\end{equation}
A direct computations shows that $\Psi_0 \in C^1(X,\mathbb{R})$ and $\Psi_\infty \in C^1(Y,\mathbb{R})$. Furthermore, associated with  $\mathcal{N}$ and $\mathcal{N}_\infty$, we consider the levels $d_0$ and $d_\infty$ given by
	$$d_0:=\inf_{u \in \mathcal{N}} I(u)\,\,\, \text{and}\,\,\, d_\infty:=\inf_{u \in \mathcal{N}_\infty} I_\infty(u).$$

	The next result presents an important property of the sets $\mathcal{N}$ and $\mathcal{N}_\infty$ that is crucial in our approach
	\begin{proposition}\label{07}
		The sets $\mathcal{N}$ and $\mathcal{N}_\infty$  are $C^1$-manifolds with the topology of $(X,||\cdot||_X)$ and $(Y,||\cdot||_Y)$ respectively,. Furthermore, the critical points of $I|_\mathcal{N}$ and ${I_\infty}|_{\mathcal{N}_\infty}$ are critical points of $I$ and $I_\infty$ respectively
	\end{proposition}	
	\begin{proof}
		For the first part, from (\ref{08}), it is sufficient to show that $0$ is a regular value for $\Psi_0$ and $\Psi_\infty$. Indeed, if $u\in \Psi_0^{-1}(\{0\})$, then 
		$$\Psi'_0(u)u=I'(u)u-\int_{\Omega}|u|^2= -\int_{\Omega}|u|^2<0,$$
		since $u\neq 0$. Consequently, $\Psi_0'(u)\neq 0$ and $0$ is a regular value of $\Psi_0$. A similar reasoning shows that $0$ is also a regular value of $\Psi_\infty$. 
		
		Now, note that if $u\in\mathcal{N}$ is a critical point of $I|_{\mathcal{N}}$, then it holds
		$$I'(u)=\lambda \Psi'_0(u),$$
		for some $\lambda\in\mathbb{R}$. So, one can see that
		$0=\lambda\Psi_0'(u)u,$
		which implies that $\lambda=0$ and $I'(u)=0,$ because $\Psi_0'(u)u<0$ for $u\in \mathcal{N}$. In a similar way, the result follows for ${I_\infty}|_{\mathcal{N}_\infty}.$
	\end{proof}	
The last proposition yields that a critical point of $I|_\mathcal{N}$ is a point $u \in X$ such that
$$
||I'(u)||_*:=\min_{\lambda \in \mathbb{R}}\,||I'(u)-\lambda\Psi_0'(u)||=0. \quad (\, \mbox{See \cite[Section 5.3]{Willem}} \,)
$$
Analogously, we define a critical point of ${I_\infty}|_{\mathcal{N}_\infty}$.

	\begin{remark}
		\rm{Note that in the preceding proposition, it is crucial the fact that in our approach, in view of the topology induced by the spaces $X$ and $Y$, the energy functionals $I$ and $I_\infty$ are of $C^1$ class. This fact it is not verified if we consider, for example, $I$ and $I_\infty$ with the usual topology of $H_0^1(\Omega)$ and $H^1(\mathbb{R}^N)$.}
	\end{remark}
	
	In the next result, we point out an important property related with the sets $\mathcal{N}$ and $\mathcal{N}_\infty$ that will be explored later on.
	
	\begin{proposition}\label{054}
		There exist $\rho_1$, $\rho_2>0$ such that
		$$\rho_1 \leq ||u||_X,\,\,\,\forall\, u \in \mathcal{N}$$
		and
		$$\rho_2 \leq ||u||_Y,\,\,\,\forall\, u \in \mathcal{N}_\infty.$$
	\end{proposition}
	\begin{proof}
		In fact, for $u \in \mathcal{N}$ it holds
		$$0<||u||_{H_0^1(\Omega)}^2\leq||u||_{H_0^1(\Omega)}^2+\int_{\Omega}F'_1(u)u=\int_{\Omega}F'_2(u)u\leq ||u||_{H_0^1(\Omega)}^p,$$
		with $p\in (2,2^*]$. Using the embedding $X \hookrightarrow H_0^1(\Omega)$, one gets
		$$0<1\leq||u||_{H_0^1(\Omega)}^{p-2}\leq C||u||_X^{p-2},$$
		for a convenient $C=C(p)>0$. Thus, the first part of the result follows by setting $\rho_1:=(C^{-1})^{\frac{1}{p-2}}$. The second part of the lemma is proved with a similar argument. 
	\end{proof}

	From now on, let us designate by $u_\infty$ a positive ground state solution of $(P_\infty)$ that can be assumed radial, that is,   
	$$
	I_\infty(u_\infty)=d_\infty>0 \quad \mbox{and} \quad I'_\infty(u_\infty)=0. (\, \mbox{See Theorem \ref{1}}\,)
	$$
	
	The next result relates the levels $d_0$ and $d_\infty$.
	
	\begin{lemma}\label{040} It holds $d_0=d_\infty$.
	\end{lemma}
	\begin{proof}
		Fix $\rho>0$ the smallest positive number such that $\mathbb{R}^N \setminus \Omega \subset B_\rho(0)$. Now, let $\phi \in C^\infty(\mathbb{R}^N)$ satisfying 
		$$\left\{\begin{aligned}
		&\phi(x)=0,  \quad \,& x \in B_\rho(0)\\
		&\phi(x)=1, & x\in B_{2\rho}(0)^c,
		\end{aligned}
		\right.
		$$
		with $0\leq\phi\leq1$. Take $(y_n) \subset \mathbb{R}^N$ with $|y_n|\rightarrow \infty$ and set
		$$\phi_n(x):=\phi(x)u_\infty(x-y_n).$$
		For each $n\in\mathbb{N}$, fix $t_n>0$ of a such way that $t_n\phi_n \in \mathcal{N}$. Thereby, 
		\begin{equation}\label{05}
		d_0\leq I(t_n\phi_n)=I_\infty(t_n\phi_n), \quad \forall n \in \mathbb{N}. 
		\end{equation}
		Note that, from the Lebesgue's Dominated Convergence Theorem,
		\begin{equation}\label{04}
		\phi(\cdot+y_n)u_\infty \longrightarrow u_\infty.
		\end{equation}
		Our next step is proving that $t_n \rightarrow 1$. To see why, firstly we recall that $t_n\phi_n \in \mathcal{N}$ leads to
		\begin{equation} \label{tn=0}
		\int_{\mathbb{R}^N}(|\nabla (t_n\phi_n)|^2+|(t_n\phi_n)|^2)= \int_{\mathbb{R}^N}(t_n\phi_n)^2\log (|t_n\phi_n|^2).
		\end{equation} 
	This combined with (\ref{3}) gives 
		$$\int_{\mathbb{R}^N}(|\nabla (\phi_n)|^2+|(\phi_n)|^2)=2\int_{\mathbb{R}^N}(F_2(\phi_n)-F_1(\phi_n))+\log t_n^2\int_{\mathbb{R}^N}\phi_n^2.
		$$
		Using (\ref{04}) and the invariance by translation of $\mathbb{R}^N$, one finds
		$$
		\int_{\mathbb{R}^N}F_i(\phi_n)\longrightarrow \int_{\mathbb{R}^N}F_i(u_\infty) \,\, \mbox{for} \,\, i\in\{1,2\} \quad \text{and} \quad \int_{\mathbb{R}^N}|\phi_n|^2\longrightarrow \int_{\mathbb{R}^N} |u_\infty|^2. 
		$$
		Gathering the limits above with (\ref{04}), one sees that $(t_n)$ is a bounded. So, we may assume that $t_n\rightarrow t_0\geq 0$. If $t_0=0$, the equality \eqref{tn=0} gives a contradiction. Therefore, it holds $t_0>0$ and, from the Lebesgue's Theorem,
		$$\int_{\mathbb{R}^N}(|\nabla (t_0u_\infty)|^2+|(t_0u_\infty)|^2)= \int_{\mathbb{R}^N}|t_0u_\infty|^2\log (t_0u_\infty|^2),
		$$
		showing that $t_0=1$, that is, $t_n \to 1$ as $n \to +\infty$.  Using this limit together (\ref{05}), we arrive at 
		$$d_0\leq \lim I_\infty(t_n\phi_n) = I_\infty(u_\infty)=d_\infty.$$
		As $X \subset Y$, the reverse inequality follows directly of the definition of $I_\infty$, by noting that the condition $I'(u)u=0$ also implies $I_\infty'(u)u=0$.
	\end{proof}	
	
	Next, we establish the nonexistence of \textit{ground state solution} for $(P_0)$, i.e.,  we are going to prove that does not exist a positive solution $u_0$ of $(P_0)$ such that $I(u_0)=d_0$.
	
	\begin{theorem}\label{034} The problem $(P_0)$ has no ground state solution.
	\end{theorem}
	
	\begin{proof}
Seeking for a contradiction, assume that $(P_0)$ has a positive ground state solution $w\in X$. Then, 
$$
I'(w)=0 \quad \mbox{and} \quad I(w)=d_0.
$$ 
Let $v$ be the null extension of $w$, i.e., $v(x)=w(x)$ for $x \in \Omega$ and $v(x)=0$ otherwise. It follows that $I_\infty'(v)v=I'(w)w=0$, and by Lemma \ref{040}, $I_\infty(v)=I(w)=d_0=d_\infty$. Therefore, $v \in \mathcal{N}_\infty$ is a critical point for $I_\infty|_{\mathcal{N}_\infty}$, and so, $v$ is a critical point of $I_\infty$. As made in \cite[Section 3.1]{d'Avenia}, by using a suitable version of the maximum principle found in \cite{Vazquez}, one deduces that $v>0$ in whole $\mathbb{R}^N$, which is absurd because $v=0$ in $\mathbb{R}^N \setminus \Omega$, finishing the proof. 
	\end{proof}
	
In order to prove our next proposition, we recall the logarithmic Sobolev inequality below, which can be found in \cite[Section 1]{Alves-de Morais}: 
	\begin{equation}\label{042}
		\int_{\mathbb{R}^N}|u|^2 \log |u|^2\,dx\leq \frac{a^2}{\pi}||\nabla u||_2^2+(\log ||u||_2^2 - N(1+\log a))||u||_2^2,\,\,\,\forall u \in H^1(\mathbb{R}^N),
	\end{equation}	
	where $a>0$ is a fixed positive constant. 
	
	Let us recall that a $(PS)_c$ sequence for $I|_{\mathcal{N}}$ is a sequence $(u_n) \subset \mathcal{N}$ such that
	$$||I'(u_n)||_*\rightarrow 0\,\,\,\,\text{and}\,\,\,\,I(u_n)\rightarrow c.$$
	\begin{lemma}\label{limitasequencia}
	If $(u_n)$ is a $(PS)_c$ sequence for $I|_{\mathcal{N}}$, then $(u_n)$ is bounded in $X$.  
	\end{lemma}
	\begin{proof}
		Let $(u_n)$ be a $(PS)_c$ sequence for $I|_\mathcal{N}$. Since $I'(u_n)u_n=0$, one has
		\begin{equation}\label{043*}
			c+o_n(1)= I(u_n)-\frac{1}{2}I'(u_n)u_n= \frac{1}{2}\int_{\Omega}|u_n|^2,
		\end{equation}
		and so,
		$$\int_{\Omega} |u_n|^2 \leq C, \quad \forall n \in \mathbb{N},$$
		for a convenient $C>0$. Applying the logarithmic inequality for some $a\approx 0^+$, we derive that 
		$$\int_{\mathbb{R}^N} |v|^2\log |v|^2 \leq \frac{1}{2}||\nabla v||_2^2+C(\log ||v||_2^2+1)||v||_2^2, \quad v\in H^1(\mathbb{R}^N), $$
		which leads to
		$$\int_{\Omega} |u_n|^2\log |u_n|^2 \leq \frac{1}{2}||\nabla u_n||_2^2+C,$$
		for some $C>0$ independent of $n$. Therefore, by (\ref{042}), there are $C_1$, $C_2>0$ such that
		$$C_1 \geq \frac{1}{2} ||u_n||_{H_0^1(\Omega)}^2-\frac{1}{2}\int_{\Omega} |u_n|^2\log |u_n|^2\geq C_2 ||u_n||_{H_0^1(\Omega)}^2,$$
		showing that 
		\begin{equation}\label{044*}
			\sup_{n \in \mathbb{N}}\,||u_n||_{H_0^1(\Omega)}^2< \infty.
		\end{equation}
		The definition of $I$ gives
		$$ \int_{\Omega}F_1(u_n)=I(u_n)-||u_n||_{H_0^1(\Omega)}^2 +\int_{\Omega} F_2(u_n).$$
		Hence, by \eqref{043*} and \eqref{044*},  
		\begin{equation}\label{045*}
			\sup_{n \in \mathbb{N}}\,\int_{\Omega} F_1(u_n)<\infty.
		\end{equation}
		The sentences (\ref{044*})  and (\ref{045*}) guarantee that $(u_n)$ is a bounded sequence in $X$.
	\end{proof}

	Our next result is an important compactness lemma that describes the behavior of $(PS)_c$ sequences for $I|_{\mathcal{N}}$.  
	\begin{lemma}\label{031}
		Let $(u_n)$ be a $(PS)_c$ sequence for $I|_{\mathcal{N}}$ with $u_n \rightharpoonup u_0$. Then, going to a subsequence if necessary, either
		\begin{description}
			\item [$i)$] $u_n \rightarrow u_0$ in $X$, or
			\item [$ii)$] There exist $k\in\mathbb{N}$ and $k$ sequences  $(u_n^j)_{n\in\mathbb{N}}$, $u_n^j\in Y$, with
			$$u_n^j\rightharpoonup u_j$$
			and $u_j$ nontrivial solutions of $(P_\infty)$, $j\in \{1,...,k\}$. Furthermore, it holds
			\begin{equation*}
			||u_n||_{H_0^1(\Omega)}^2\rightarrow ||u_0||_{H_0^1(\Omega)}^2+\sum_{j=1}^{k}||u_j||_{H^1(\mathbb{R}^N)}^2\,\,\,\,\text{and}\,\,\,\,I(u_n)\rightarrow I(u_0)+\sum_{j=1}^{k} I_\infty(u_j).
			\end{equation*}
		\end{description}	
	\end{lemma}	
	\begin{proof}
		Initially, for a convenient sequence of real numbers $(\lambda_n)$, we must have 
		\begin{equation}\label{09}
		I'(u_n)=\lambda_n\Psi_0'(u_n)+o_n(1).
		\end{equation}
	As $I'(u_n)u_n=0$, one gets
		$$\lambda_n\Psi_0'(u_n)u_n=o_n(1).$$
		From this information, we claim that $\lambda_n=o_n(1)$. Indeed, notice that $|\Psi_0'(u_n)u_n|\nrightarrow 0$, otherwise we would have
		$$\Psi_0'(u_n)u_n=\int_{\Omega}|u_n|^2=o_n(1),$$
		and so, since $(u_n)$ is a bounded  sequence in $X$, by interpolation, it follows that 
		$$
		||u_n||_p=o_n(1), \quad \forall p \in (2,2^*).
		$$
		This combines with $(P_2)$ to give
		$$
		\displaystyle{\int_{\Omega}}F'_2(u_n)u_n=o_n(1).
		$$
		Now, the limit above together with the fact that $I'(u_n)u_n=0$ leads to
		$$\int_{\Omega}(|\nabla u_n|^2+2|u_n|^2)+\int_{\Omega}F'_1(u_n)u_n=o_n(1).$$
	Since $F_1$ is convex with $F_1(0)=0$, we know that $F'_1(s)s \geq F_1(s)$ for all $s \in \mathbb{R}$. Then, we can infer that  
		$$\int_{\mathbb{R}^N}(|\nabla u_n|^2+2|u_n|^2)+\int_{\Omega}F_1(u_n)=o_n(1).$$
Using the fact that $F_1 \in (\Delta_{2})$, the last limit yields $u_n \to 0$ in $X$, which contradicts the fact that $u_n \in \mathcal{N}$ in view of the Proposition \ref{054}. So, it follows that $|\Psi_0'(u_n)u_n|\nrightarrow 0$ and $\lambda_n=o_n(1)$. By (\ref{09}), since $(u_n)$ is a bounded sequence, it holds $I'(u_n)\rightarrow 0$, that is, the sequence $(u_n)$ is a $(PS)_c$ sequence for $I$. In addition, accounting that $u_n\rightharpoonup u_0$ and the growth conditions on $F_1$ and $F_2$, we deduce that $I'(u_0)v = 0$, for any $v \in X$, implying that $u_0$ is a solution of $(P_0)$.
		
		From now on, inspired in the ideas of \cite{Benci-Cerami}, we set	
		$$
		\psi_n^1(x):=	\left\{\begin{aligned}&u_n-u_0,\,\,\, &x\in \Omega\\
		&\,\,\,\quad0,\,\,\, &x \in \mathbb{R}^N \setminus \Omega.
		\end{aligned}
		\right.
		$$
		A direct verification shows that $\psi_n^1\rightharpoonup 0$ in $X$. 
	In \cite{Alves-Carriao-Medeiros, Benci-Cerami}, it was proved that $(\psi_n^1|_{\Omega})$ is a $(PS)$ sequence for $I_\infty|_{H_0^1(\Omega)}$ with
		\begin{equation}\label{010}
		I_\infty(\psi_n^1)=I(u_n)-I(u_0)+o_n(1).
		\end{equation}
		However, since we are working with a logarithmic nonlinearity, we are not able to show that  $(\psi_n^1|_{\Omega})$ is also a $(PS)$ sequence.  In our case we will prove that a more weak condition occurs.  More precisely, the following properties hold: \\
		$i)$ $I_\infty(\psi_n^1)=I(u_n)-I(u_0)+o_n(1);$\\
		$ii)$ Let $\phi \in C_0^{\infty}(\Omega)$ with $||\phi||_Y\leq 1$ and, for each $y \in \mathbb{R}^N$, define $\phi^{(y)}(x)=\phi(x+y)$ for all $x \in \mathbb{R}^N$. Then, 
	$$\sup_{y \in \mathbb{R}^N} \|I'_\infty(\psi_n^1)\|||\phi^{(y)}||_Y = o_n(1).$$\\
		\\
		\textbf{Verification of $i)$} \, By simplicity,  in what follows $\psi_n^1$ denotes $\psi_n^1|_{\Omega}$. The definition of  $\psi_n^1$ gives $I_\infty(\psi_n^1)=I(\psi_n^1)$, then by a simple computation, the Lemma \ref{Brezis-Lieb} guarantees that $i)$ holds.  \\
	
	\noindent 	\textbf{Verification of $ii)$} \, First of all, note that 
		\begin{equation}\label{025}
		I_\infty'(\psi_n^1)\phi^{(y)}=\int_{\Omega}(\nabla \psi_n^1\nabla \phi^{(y)}+2\psi_n^1\phi^{(y)})+\int_{\Omega}F'_1(\psi_n^1)\phi^{(y)}-\int_{\Omega}F'_2(\psi_n^1)\phi^{(y)}.
		\end{equation}
		
		In order to prove the item $ii)$, we will need to show the following claim 
		
		\begin{claim}\label{024}
			$$\sup_{y \in \mathbb{R}^N}\int_\Omega |F'_i(u_n-u_0)-(F'_i(u_n)-F'_i(u_0))||\phi^{(y)}|= o_n(1), \quad \mbox{for} \quad i \in \{1,2\}.$$
		\end{claim}	
		
		In the proof of the claim above, we adapt some ideas presented in \cite[Proof of (3.39)]{Alves-Ji2}. In what follows, we  will only show that the claim for function $F_1$, because the proof for $F_2$ follows by using similar arguments (see also \cite[Lemma 3.1]{Alves-Carriao-Medeiros}). 
		
		Given $\varepsilon>0$ and $r \in (1,2)$,  the definition of $F_1$ guarantees that there is $t_0>0$ small enough such that
		\begin{equation}\label{015}
		|F'_1(t)| \leq \varepsilon |t|^{r-1},\,\,\, |t|\leq  2t_0. 
		\end{equation}
		On the other hand, note that it is possible to get $t_1>t_0$ large enough such that
		\begin{equation}\label{016}
		|F_1'(t)| \leq \varepsilon |t|^{2^*-1},\,\,\, |t|\geq t_1-1,
		\end{equation}
		as well as
		\begin{equation}\label{017}
		|F_1'(t)-F_1'(s)| \leq \varepsilon|t_0|^{r-1}, \,\,\,|t-s|\leq s_0, \,\,\,\text{and}\,\,\, |t|,|s|\leq t_1+1,
		\end{equation}
		for some $s_0>0$ small enough. Therefore,  
		\begin{equation}\label{018}
		|F_1'(t)| \leq C_\varepsilon |t|^{r-1}+\varepsilon|t|^{2^*-1},\,\,\ t\in \mathbb{R},
		\end{equation} 
		for some $C_\varepsilon>0$. Now, fixing $R>0$ of such way $B^c_R(0)\subset \Omega$ and using then fact that $F_1$ has a subcritical growth, it is easy to prove that   
		$$
		\int_{B_R(0)\cap \Omega}\mid F_1'(u_n-u_0)-(F'_1(u_n)-F'_1(u_0))\mid \mid\phi^{(y)}\mid=o_n(1), \quad \mbox{uniformly in} \quad y \in \mathbb{R}^N. 
		$$
		Our next step is to estimate the integral below 
		$$
		\int_{B^c_R(0)\cap \Omega}\mid F_1'(u_n-u_0)-(F'_1(u_n)-F'_1(u_0))\mid \mid\phi^{(y)}\mid.
		$$ 
		Fix $\varepsilon>0$. From (\ref{018}), since $R>0$ can be chosen large enough, one has 
		\begin{equation}\label{022}
		\begin{aligned}
		\int_{B^c_R(0)\cap\Omega} \mid F'_1(u_0) \mid\mid\phi^{(y)} \mid&\leq C_\varepsilon\int_{B^c_R(0)\cap\Omega}\mid u_0 \mid^{r-1}\mid \phi^{(y)} \mid+\varepsilon\int_{B^c_R(0)\cap\Omega}\mid u_0 \mid^{2^*-1}\mid \phi^{(y)} \mid\\
		&\leq C(||u_0||_2^{r-1}||\phi^{(y)}||_{\frac{2}{3-r}}+||u_0||_{2^*}^{2^*-1}||\phi^{(y)}||_{2^*})\\
		&\leq \varepsilon C||\phi^{(y)}||_Y,
		\end{aligned} 
		\end{equation}
	where $C$ does not depend on $y \in \mathbb{R}^N$. Setting 
		$$A_n:=\{x\in B^c_R(0) ;\,|u_n(x)|\leq t_0 \}$$
		and
		$$B_n:=\{x\in B^c_R(0)  ;\,|u_n(x)|\geq t_1 \},$$
		we have by (\ref{015}),
		\begin{equation}\label{019}
		\begin{aligned}
		\int_{A_n\cap [|u_0|\leq \delta]}\mid F'_1(u_n-u_0)-F'_1(u_n)\mid \mid \phi^{(y)} \mid &\leq\varepsilon\int_{A_n\cap [|u_0|\leq \delta]}(\mid u_n-u_0 \mid^{r-1}\mid \phi^{(y)} \mid+\mid u_n \mid^{r-1}\mid \phi^{(y)} \mid\\
		&\leq \varepsilon \, C||\phi||_Y, 
		\end{aligned}
	\end{equation}
	where $C$ does not depend on $y \in \mathbb{R}^N$. Here, we have explored the fact that $|\text{supp}\,\phi^{(y)}|=|\text{supp}\,\phi|$ for any $y\in \mathbb{R}^N$.  In a similar way, by using (\ref{016}), 
		\begin{equation}\label{020}
		\int_{B_n\cap [|u_0|\leq \delta]}\mid F'_1(u_n-u_0)-F'_1(u_n)\mid \mid \phi^{(y)} \mid\leq \varepsilon C||\phi||_Y.
		\end{equation}
		Next, let us consider $C_n:=\{x \in B^c_R(0);\,t_0\leq |u_n(x)| \leq t_1\}$. Accounting that $(u_n)$ is a bounded sequence in $X$, we derive that
		$$M:=\sup_{n\in \mathbb{N}} |C_n |< \infty.$$
		Thereby, by (\ref{017}),
		\begin{equation}\label{021}
		\int_{C_n\cap [|u_0|\leq \delta]}\mid F'_1(u_n-u_0)-F'_1(u_n)\mid \mid \phi^{(y)} \mid \leq t_0^{r-1}\varepsilon|C_n |^{1/2}||\phi^{(y)}||_2\leq\varepsilon C||\phi||_Y,
		\end{equation}
		for a convenient $C$ independent of $\varepsilon$ and $y \in\mathbb{R}^N$. From (\ref{019}), (\ref{020}) and (\ref{021}), 
		\begin{equation}\label{023}
		\int_{B^c_R(0)\cap [|u_0|\leq \delta]}\mid F'_1(u_n-u_0)-F'_1(u_n)\mid \mid \phi^{(y)} \mid\leq \varepsilon C||\phi^{}||_Y.
		\end{equation}
		Now, we are going to analyze the case that $|u_0|>\delta$. The boundedness of $(u_n)$ in $X$ together with the inequality (\ref{018}) give 
		$$\begin{aligned}
		\int_{B^c_R(0)\cap [|u_0|> \delta]}\mid F'_1(u_n-u_0)-F'_1(u_n)\mid \mid \phi^{(y)} \mid&\leq C_\varepsilon\int_{B^c_R(0)\cap [|u_0|> \delta]}(\mid u_n-u_0 \mid^{r-1}\mid \phi^{(y)} \mid+\mid u_n \mid^{r-1}\mid \phi^{(y)} \mid\\
		&+\varepsilon C||\phi||_Y,
		\end{aligned}
		$$
	where $C$ is independent of $\varepsilon$ and $y$. Since $u_0 \in X \subset H_0^1(\Omega)$, one has 
		$$
		|B^c_R(0)\cap [|u_0|> \delta]|\longrightarrow 0, \quad \mbox{as} \quad R\rightarrow +\infty.
		$$
		Thereby,
		$$ \begin{aligned}
		&C_\varepsilon\int_{B^c_R(0)\cap [|u_0|> \delta]}(\mid u_n-u_0 \mid^{r-1}\mid \phi^{(y)} \mid+\mid u_n \mid^{r-1}\mid \phi^{(y)} \mid\leq\\
		&\leq C_\varepsilon (||(u_n-u_0||_{2^*}^{r-1} + ||u_n||_{2^*}^{r-1})||\phi||_{2^*}|B_R(0)^c\cap [|u_0|> \delta]|^{(2^*-r)/2^*}\leq\\
		&\leq\varepsilon C||\phi||_Y,
		\end{aligned}
		$$
	for $R>0$ large enough and $C$ independent of $\varepsilon$ and $y$. Using the last information together with (\ref{022}) and  (\ref{023}), one finds
		$$
		\sup_{y \in \mathbb{R}^N}\int_{B^c_R(0)\cap \Omega}\mid F_1'(u_n-u_0)-(F'_1(u_n)-F'_1(u_0))\mid \mid\phi^{(y)}\mid\leq \varepsilon C||\phi||.
		$$ 
		Since $\varepsilon$ is an arbitrary positive number, the last inequality with $||\phi||_Y\leq1$ ensures that the Claim \ref{024} is valid for the function $F_1$ and this finishes the proof of the claim. 
		
		Now, we are ready to show the item $ii)$. In fact, fix $\phi\in C_0^\infty(\Omega)$. So, by (\ref{025}), 
		$$\begin{aligned}
		[I'_\infty(\psi_n^1)-(I'(u_n)-I'(u_0))](\phi^{(y)})&=\int_\Omega [F'_1(u_n-u_0)-(F'_1(u_n)-F'_1(u_0))]\phi^{(y)}\,+\\
		&+\int_\Omega [F'_2(u_n-u_0)-(F'_2(u_n)-F'_2(u_0))]\phi^{(y)}.
		\end{aligned}
		$$
		Hence, by Claim \ref{024}, 
		$$\sup_{y \in \mathbb{R}^N}|I'_\infty(\psi_n^1)-(I'(u_n)-I'(u_0))|||\phi^{(y)}||_Y=o_n(1),$$
	from where it follows that  
		$$\sup_{y \in \mathbb{R}^N}\|I'_\infty(\psi_n^1)\|\,||\phi^{(y)}||_Y=o_n(1),$$
		and the item $ii)$ is proved. 
		If $\psi_n^1 \rightarrow 0$, then the proof would be finished. Thereby, in order to get the desired result, let us consider that 
		\begin{equation}\label{026}
		\psi_n^1\nrightarrow 0\,\,\,\text{in}\,\,\,Y.
		\end{equation}
		In this way, we can prove that the following claim holds
		\begin{claim}\label{029}
			There exist $\lambda_0>0$ and $n_0\in \mathbb{N}$ such that
			$$I_\infty(\psi_n^1)\geq\lambda_0,\,\,\, \forall n\geq n_0.$$
		\end{claim}	 
		\noindent Otherwise, considering a subsequence of $(\psi_n^1)$ if necessary, we would have
		$$ I_\infty(\psi_n^1) = o_n(1).$$
		Now, recalling that  
		$$F_2'(t)t-F_1'(t)t=t^2\log t^2+t^2,\,\,\,t\in\mathbb{R}$$
		the same arguments explored in the proof of item $i)$ ensure that 
		$$I'_\infty(\psi_n^1)\psi_n^1=I'(u_n)u_n-I'(u_0)u_0=o_n(1),$$
	and so, 
		$$I_\infty(\psi_n^1)=I_\infty(\psi_n^1)-\frac{1}{2}I'_\infty(\psi_n^1)\psi_n^1+o_n(1)=\frac{1}{2}\int_{\mathbb{R}^N}|\psi_n^1|^2+o_n(1).$$
		Consequently, one finds $ \displaystyle{\int_{\mathbb{R}^N}}|\psi_n^1|^2=o_n(1)$, and by interpolation, $\displaystyle{\int_{\mathbb{R}^N}}|\psi_n^1|^p=o_n(1)$. So, the growth condition on $F_2$ allow us to conclude that
		$$\int_{\mathbb{R}^N}F_2'(\psi_n^1)\psi_n^1=o_n(1).$$
		From the computations above, one has
		$$  ||\psi_n^1||_{H_1(\mathbb{R}^N)}^2+\int_{\mathbb{R}^N}F'_1(\psi_n^1)\psi_n^1=o_n(1),
		$$ 
		which contradicts (\ref{026}). Then, the claim is proved.
		
		Now, lets us consider a decomposition of $\mathbb{R}^N$ into unit hypercubes $B_i$ with vertices having integer coordinates and set 
	$$d_n:=\max_{i\in\mathbb{N}} ||\psi_n^1||_{L^p(B_i)},$$
		for a fixed $p\in (2,2^*)$.
	\begin{claim} \label{027} There exist $\lambda_1>0$ and $n_1 \in \mathbb{N}$ such that
			$$d_n\geq \lambda_1,\,\,\, \forall n\ge n_1.$$
	\end{claim}
\noindent Arguing as in the last claim, 
		$$I'_\infty(\psi_n^1)\psi_n^1=I'(u_n)u_n-I'(u_0)u_0=o_n(1),$$
		and so
		$$||\psi_n^1||_{H^1(\mathbb{R}^N)}^2 + \int_{\mathbb{R}^N}F'_1(\psi_n^1)\psi_n^1=\int_{\mathbb{R}^N}F'_2(\psi_n^1)\psi_n^1+o_n(1).$$
	By (\ref{2}), 
		$$
		C\left(||\psi_n^1||_{H^1(\mathbb{R}^N)}^2 + \int_{\mathbb{R}^N}F_1(\psi_n^1)\right)\leq \int_{\mathbb{R}^N}F'_2(\psi_n^1)\psi_n^1+o_n(1),
		$$
		for some constant $C>0$. Combining this inequality with $(P_2)$, one finds 
		$$\begin{aligned}
		I_\infty(\psi_n^1)&=\frac{1}{2}||\psi_n^1||_{H^1(\mathbb{R}^N)}^2 + \int_{\mathbb{R}^N}F_1(\psi_n^1)-\int_{\mathbb{R}^N}F_2(\psi_n^1)\leq\\
		&\leq C\int_{\mathbb{R}^N}|\psi_n^1|^p+o_n(1)=C\sum_{i \in \mathbb{N}}||\psi_n^1||_{L^p(B_i)}^p+o_n(1).
		\end{aligned}
		$$
		Since each $B_i$ is a unit hypercube of $\mathbb{R}^N$, there is a constant $\tilde{C}>0$ independent of $i$ such that
		\begin{equation}\label{070}
		||\psi_n^1||_{L^p(B_i)}\leq \tilde{C}||\psi_n^1||_{H^1(B_i)},\,\,\,\forall i \in \mathbb{N}.
		\end{equation}
		Hence, modifying $\tilde{C}>0$ if necessary, it holds
	$$
	I_\infty(\psi_n^1)+o_n(1)\leq C d_n^{p-2}\sum_{i\in \mathbb{N}}||\psi_n^1||_{H^1(B_i)}^2\leq Md_n^{p-2},
	$$
	for some $M>0$. Now, we apply Claim \ref{029} to get the desired result. 
		
		Hereafter, for our goals, let us consider $y_n^1$ the center of $B_i$ in such way that 
		$$d_n=||\psi_n^1||_{L^p(B_i)}.$$
		In this way, one can see that, by taking a subsequence, $|y_n^1|\rightarrow \infty$. Otherwise, for some $R>0$ large enough we must have 
		$$\int_{B_R(0)}|\psi_n^1|^p\geq\int_{B_i}|\psi_n^1|^p=d_n^p\geq \lambda_1^p>0,$$
		which is a contradiction, because the weak limit $\psi_n^1 \rightharpoonup 0$ in $Y$ implies that
		$$
		\int_{B_R(0)}|\psi_n^1|^p\longrightarrow 0.
		$$
	Thereby, we may assume that $|y_n^1|\rightarrow \infty$. 
		
		Notice that, by the invariance of translations of $\mathbb{R}^N$, we conclude that $(\psi_n^1(\cdot+y_n^1))$ is bounded in $Y$. Then, for some $u_1 \in Y$,
		\begin{equation}\label{028}
		\psi_n^1(\cdot+y_n^1)\rightharpoonup u_1 \,\,\,\text{in}\,\,\,Y.
		\end{equation}
		
		Our next step is to prove that $u_1$ is a nontrivial solution of $(P_\infty)$.
		\begin{claim} The function $u_1$ is a nontrivial solution of $(P_\infty)$.
		\end{claim}
		Initially, let us prove that $u_1\neq 0$. To see why, let us denote by $B_0$ the unit hypercube of $\mathbb{R}^N$ centered at the origin. Then, by the Claim \ref{027},
		$$\int_{B_0}|\psi_n^1(\cdot+y_n^1)|^p=\int_{B_i}|\psi_n^1|^p= d_n^p\geq \lambda_1^p>0.$$
		Observe that, by (\ref{028}), $\psi_n^1(\cdot+y_n^1)\rightarrow u_1$ in $L^p(B_0)$. Hence,
		$$\displaystyle{\int_{B_0}}|u_1|^p\geq \lambda_1^p>0,$$  
		showing that $u_1\neq 0$.	
		
		Set 
		$$\Omega_n:=\{x \in\mathbb{R}^N;\,x+y_n^1\in \Omega\}.$$
		Note that, for each $v\in C_0^\infty(\mathbb{R}^N)$, we have that $\text{suppt}\, v\subset \Omega_n$ for $n$ large enough. Setting $v^{(n)}(x):= v(x-y_n^1)$, it follows that
		$$
		\text{suppt} \,v^{(n)}\subset \Omega \quad \mbox{and} \quad v^{(n)}\in H_0^1(\Omega).
		$$ 
		Taking $v\in C_0^\infty(\mathbb{R}^N)$ with $||v||_Y\leq1$, we see that $||v^{(n)}||_X=1$ and
		$$I'_\infty(\psi_n^1(\cdot+y_n^1))v=I'_\infty(\psi_n^1)v^{(n)}.$$
		Thus, by item $ii)$, $I'_\infty(\psi_n^1(\cdot+y_n^1))v=o_n(1)$. On the other hand, standard arguments involving the weak convergence of $(\psi_n^1(\cdot+y_n^1))$ yield
		$$I'_\infty(\psi_n^1(\cdot+y_n^1))v=I'_\infty(u_1)v.$$
		By gathering these information, we derive that $I'_\infty(u_1)v=0$, then $u_1$ is a nontrivial critical point of $I_\infty$, and so, $u_1$ is a solution of $(P_\infty)$.
		
		Define $\psi_n^2:=(\psi_n^1(\cdot+y_n^1)-u_1)$. If $\psi_2^n \rightarrow 0$, then the proof is finished. Otherwise, we use the fact that $\psi_n^2 \rightharpoonup 0$ and the ideas explored above to find a unbounded sequence $(y_n^2)$ of $\mathbb{R}^N$ and to produce $u_2\in Y$ a nontrivial solution of $(P_\infty)$. Continuing with this procedure, for each $j\geq 2$ it is possible to define
		$$ \psi_n^j:=\psi_n^{j-1}(\cdot+y_n^{j-1})-u_{j-1},$$
		with
		$$\left\{\begin{aligned}
		&y_n^{j-1}\rightarrow \infty\\
		&\psi_n^{j-1}\rightharpoonup u_{j-1},
		\end{aligned}
		\right.
		$$
		and $u_{j-1}$ a nontrivial solution of $(P_\infty)$. By exploring the same type of argument used in the prove of item $i)$, one can prove that
		\\
		$iii)$: $||\psi_n^j||_{H^1(\mathbb{R}^N)}^2=||u_n||_{H_0^1(\Omega)}^2-||u_0||_{H_0^1(\Omega)}^2-\displaystyle{\sum_{i=1}^{j-1}}||u_i||_{H^1(\mathbb{R}^N)}^2+o_n(1);$\\
		$iv)$: $I_\infty(\psi_n^j)=I(u_n)-I(u_0)-\displaystyle{\sum_{i=1}^{j-1}}I_\infty(u_i)+o_n(1).$ \\
		$v)$: $\displaystyle \liminf_{n \to \infty}I_\infty(\psi_n^j) >0$ for each $j \in \mathbb{N}$.
		\\
		We finish the proof by proving that the following claim holds.
		\begin{claim} \label{Claimphik} 
			There is a number $k\in\mathbb{N}$ such that $\psi_n^k\rightarrow 0$ in $Y$.
		\end{claim}
		In fact, otherwise it would be possible to get by the preceding procedure a nontrivial solution $u_j$ of $(P_\infty)$ for each $j\in \mathbb{N}$, and so,   
		$$I_\infty(u_j)\geq d_\infty=\inf_{u\in \mathcal{N}_\infty} I_\infty(u)>0,\,\,\,\forall j\in \mathbb{N}.$$
		Thus, from $iv)$, 
		$$I_\infty(\psi_n^j)\leq I(u_n)-I(u_0)-(j-1)d_\infty+o_n(1).$$
		As $(I(u_n))$ is a bounded sequence, for $j$ large enough the last inequality implies that  $\displaystyle \liminf_{n \to \infty}I_\infty(\psi_n^j)<0,$ which contradicts $v)$.  From this, the Claim \ref{Claimphik} is proved and the proof is over. 
		\end{proof}	
	
	\section{Technical Results}
	
	In this section we prove some technical results that are crucial in the proof of Theorem \ref{053}. The main goal is to prove that $I|_\mathcal{N}$  satisfies the $(PS)_c$ condition for all $c\in (d_\infty+\varepsilon,2d_\infty-\varepsilon)$, for some $\varepsilon>0$ small enough. 
	
	In the sequel, 
	$$
	\chi(t):=\left\{\begin{aligned}
	&1,\,\,\,&0\leq t \leq R;\\
	&\frac{R}{t},\,\,\,& R \leq t,
	\end{aligned}
	\right.
	$$
where $R>0$ is such that $\Omega^c \subset B_R(0)$. Next, let $\tau:Y\longrightarrow \mathbb{R}^N$ be given by
	$$\tau(u):=\int_{\mathbb{R}^N}|u|^2\chi(|x|)x$$
	and set 
	$$
P:=\{u\in X;\,u\geq 0\} \quad  \mbox{and} \quad T_0:=\{u\in \mathcal{N}\cap P; \,\tau(u)=0\}.
	$$
Employing the above notations, let us define the level
	$$c_0:= \inf_{u\in T_0}\, I(u),$$
which satisfies 
	\begin{equation}\label{030}
	d_\infty=d_0\leq c_0.
	\end{equation}
	Our first result is the following
	\begin{lemma}\label{047}
		The number $c_0$ satisfies $d_\infty < c_0$.	
	\end{lemma}	
	\begin{proof}
		Arguing by contradiction, in view of (\ref{030}), if the lemma does not hold, then it occurs 
		$$d_\infty=d_0=c_0.$$
		Thus, it is possible to take a sequence $(v_n)$ in $\mathcal{N}\cap P$ such that
		$$\tau(v_n)=0\,\,\,\text{and}\,\,\,I(v_n)\longrightarrow d_0=\inf_{u\in\mathcal{N}}\,I(u).$$
		By applying the Ekeland's Variational Principle, there is a sequence $(u_n)$ in $\mathcal{N}$ satisfying $I(u_n)\leq I(v_n)$, $||u_n-v_n||_X=o_n(1)$ and $(u_n)$ is also a $(PS)_{d_0}$ sequence for $I|_\mathcal{N}$ (see e.g. \cite[Theorem 8.5]{Willem}).
		Thanks to Lemma \ref{031}, there are $k \in \mathbb{N}$ and nontrivial solutions $u_1,...,u_k$  of $(P_\infty)$ with
		\begin{equation}\label{033}  ||u_n||_{H_0^1(\Omega)}^2\longrightarrow||u_0||_{H_0^1(\Omega)}^2+\displaystyle{\sum_{j=1}^{k}}||u_j||_{H^1(\mathbb{R}^N)}^2
		\end{equation}
	and
		\begin{equation}\label{032}
		I(u_n)\longrightarrow I(u_0)+\displaystyle{\sum_{j=1}^{k}}I_\infty(u_j),
		\end{equation}
		where $u_0$ has been chosen in a such way that $u_n \rightharpoonup u_0$ and $u_0$ is a solution of $(P_0)$. Using the fact that $d_\infty=d_0$, it holds
		$$I(u_0)+\sum_{j=1}^{k}I_\infty(u_j)\geq I(u_0)+kd_0.$$
		Since $I(u_n)\rightarrow d_0$ and $I(u_0)\geq 0$, from (\ref{032}) one has $k=0$ or $k=1$. If $k=0$, accounting (\ref{033}), we find
		$$u_n \longrightarrow u_0\,\,\,\text{in}\quad H_0^1(\Omega).$$ 
		Now, as $(u_n)$ is a $(PS)_{d_0}$ sequence for $I|_{\mathcal{N}}$ (and also for $I$) and $u_0$ is a solution of $(P_0)$, one gets
		$$
		\begin{aligned}
		||u_n||_{H_0^1(\Omega)}^2 + \int_{\Omega}F'_1(u_n)u_n&=\int_{\Omega}F'_2(u_n)u_n=\\
		&=\int_{\Omega}F'_2(u_0)u_0+o_n(1)=\\
		&=||u_0||_{H_0^1(\Omega)}^2+\int_{\Omega}F'_1(u_0)u_0+o_n(1),
		\end{aligned}
		$$
	that is, 
		$$||u_n||_{H_0^1(\Omega)}^2+\int_{\Omega}F'_1(u_n)u_n \longrightarrow ||u_0||_{H_0^1(\Omega)}^2+\int_{\Omega}F'_1(u_0)u_0. $$
		In particular, one has
		$$||u_n||_{H_0^1(\Omega)}^2 \longrightarrow ||u_0||_{H_0^1(\Omega)}^2\,\,\,\text{and}\,\,\, \int_{\Omega}F'_1(u_n)u_n \longrightarrow \int_{\Omega}F'_1(u_0)u_0,$$
		which yields that $u_n \rightarrow u_0$ in $H_0^1(\Omega)$ and $u_n \rightarrow u_0$ in $L^{F_1}(\Omega)$, since $F_1 \in (\Delta_2)$. From this, $u_n \rightarrow u_0$ in $X$, and so,  
		$$I(u_n)\longrightarrow I(u_0)=d_0,$$
	showing that $u_0$ is a ground state solution for $(P_0)$, which contradicts Theorem \ref{034}. So, $k=1$ and $u_0=0$. Otherwise, if $u_0 \neq 0$, the function $u_0$ would be a nonzero solution of $(P_0)$, and so,  
	$$ d_0=\lim I(u_n) \geq 2d_0,$$
	giving a new contradiction. By following the notation in the proof of Lemma \ref{031}, one finds
		$$\left\{\begin{aligned}
		&u_n(x+y_n^1)=\psi_n^1(x+y_n^1)\rightharpoonup u_1;\\
		&y_n^1 \rightarrow \infty.
		\end{aligned}
		\right.
		$$
		Note also that $||u_n||_{H_0^1(\Omega)}^2\rightarrow||u_1||_{H_0^1(\Omega)}^2$ and $I_\infty(u_1)=d_\infty.$ Thus, $u_1$ is a ground state solution of $(P_\infty)$.
		
		Now, on accounting of Theorem \ref{1} one can gets a contradiction by following the same ideas in \cite[Lemma 4.3]{Benci-Cerami}. For the sake of completeness, we recall some steps made in \cite[Lemma 4.3]{Benci-Cerami}. Denote, by simplicity, $y_n:=y_n^1$, 
		$$
		(\mathbb{R}^N)_n^+:=\{x\in \mathbb{R}^N;\,\langle x, y_n\rangle_{\mathbb{R}^N} > 0\},
		$$
		$$
		(\mathbb{R}^N)_n^-:=\mathbb{R}^N-(\mathbb{R}^N)_n^+,
		$$
		and
		$$w_n(x):=u_n(x)-u_1(x-y_n).$$
		The above information gives $w_n \rightarrow 0$ in $H^1(\mathbb{R}^N)$.

		By Theorem \ref{1}, without loss of generality we may assume that $u_1$ is a radially symmetric solution of $(P_\infty)$. In the same way as \cite[Lemma 4.3]{Benci-Cerami} (see also \cite[Lemma 4.3]{Alves-Ambrosio-Torres}), we derive that
		$$
		\left\{\begin{aligned}
		&u_1(x-y_n)\geq \frac{1}{2}u_1(0)>0,\,\,\,x \in B_r(y_n);\\
		&u_1(x-y_n)\rightarrow 0,\,\,\,\text{a.e}\,\,\,x\in (\mathbb{R}^N)_n^-\,\,\,\,\text{and}\,\,\,\,\int_{(\mathbb{R}^N)_n^-}|u_1(x-y_n)|^2\chi(|x|)|x| = o_n(1),
		\end{aligned}	
		\right.
		$$ 
		for some $r>0$. Then, 
		\begin{equation}\label{035}
		\langle \tau(u_1(x-y_n)),y_n/|y_n|\rangle_{\mathbb{R}^N}\geq C>0,n\geq n_0,
		\end{equation}
		for some $C>0$. On the other hand, taking into accounting that $\tau(u_1(\cdot-y_n))=\tau(u_n-w_n)$, and that $|\tau(u_n)|$, $|\tau(w_n)|=o_n(1)$, we derive that
		\begin{equation}\label{036}
		|\tau(u_1(x-y_n))|=o_n(1).
		\end{equation}
		From  (\ref{035})-(\ref{036}), we find a contradiction, finishing the proof. 
		\end{proof}
	
	Hereafter we will fix $\rho>0$ as the smallest positive number such that $\Omega^c \subset B_\rho(0).$ Let $\phi(x):=\varphi(\frac{|x|}{\rho})$, where $\varphi \in C_0^\infty([0,\infty))$ is an increasing function such that $\varphi(t)=0$, $0\leq t\leq 1$, and $\varphi(t)=1$, $t\geq 2$. Now, for each $y\in \mathbb{R}^N$, we set
	$$\psi_{y, \rho}(x):=\phi(x)u_\infty(x-y),$$
	where $u_\infty\in \mathcal{N}_\infty$ is a ground state solution of $(P_\infty)$, which is assumed to be a decreasing and radially symmetric at the origin. Finally, fix $t_{y,\rho}>0$ satisfying
	$$\phi_\rho(y):=t_{y,\rho}\psi_{y,\rho}\in \mathcal{N}_\infty.$$
Next, we prove an important property related to the mappings $\phi_\rho(y)$.
	\begin{lemma}\label{041}
		The family of mappings $(\phi_\rho(y))$ satisfies the following limits:
		\begin{enumerate}
			\item [$i)$:] $\displaystyle{\lim_{\rho\rightarrow 0}} I_\infty(\phi_\rho(y))= d_\infty$, uniformly in $y\in\mathbb{R}^N$;
			\item [$ii)$:] For each fixed $\rho>0$, it holds $\displaystyle{\lim_{|y|\rightarrow \infty}} I_\infty(\phi_\rho(|y|)) = d_\infty$.
		\end{enumerate}	
	\end{lemma}
	\begin{proof}
		\textbf{Verification of $i)$:} From the definition of $\psi_{y, \rho}$ and the properties of $u_\infty$ (see Theorem \ref{1} above), for each fixed $p\in[2,2^*]$, one has
		$$\begin{aligned}
		||\psi_{y,\rho}- u_\infty(\cdot-y)||_p^p &\leq C\int_{B_{2\rho}(0)}|u_\infty(\cdot-y)|^p\\
		&\leq C\int_{B_{2\rho}(0)}|u_\infty(0)|^p\\
		&\leq \tilde{C}\rho^N=o_\rho(1),\,\,\,\, \forall y \in \mathbb{R}^N.
		\end{aligned}
		$$	
		Similarly, since $N\geq 3$,
		$$\begin{aligned}
		||\nabla (\psi_{y,\rho}- u_\infty(\cdot-y))||_2^2&\leq C\int_{B_{2\rho}(0)}|\nabla \phi|^2|u_\infty(\cdot-y)|^2+C\int_{B_{2\rho}(0)}|\phi(x)-1|^2|\nabla u_\infty(\cdot-y)|^2\\
		&\leq C_1\rho^N + C_2 \rho^{N-2},\,\,\,\,\forall y\in \mathbb{R}^N.
		\end{aligned}
		$$
		Hence,
		$$||\psi_{y,\rho}||_p\longrightarrow|| u_\infty(\cdot-y)||_p \,\,\,\text{as}\,\,\,\rho\rightarrow 0,$$
		as well as
		$$||\psi_{y,\rho}||_{H^1(\mathbb{R}^N)}\longrightarrow|| u_\infty(\cdot-y)||_{H^1(\mathbb{R}^N)},\,\,\,\text{as}\,\,\,\rho\rightarrow 0,$$
		uniformly in $y \in \mathbb{R}^N$. From this, 
		$$\int_{\mathbb{R}^N} F_2(\psi_{y,\rho}) \longrightarrow \int_{\mathbb{R}^N}F_2(u_\infty),\,\,\,\text{as}\,\,\, \rho \rightarrow 0,$$
		uniformly in $y \in \mathbb{R}^N$.
		Now, using the definition of $\psi_{y,\rho}$, one gets
		\begin{equation}\label{038}
		\int_{\mathbb{R}^N}|F_1(\psi_{y, \rho}) - F_1(u_\infty(\cdot-y))| =\int_{B_\rho(0)}|F_1(\psi_{y, \rho}) - F_1(u_\infty(\cdot-y))|.
		\end{equation}	
		By the mean value theorem, 
		\begin{equation}\label{037}
		\int_{B_\rho(0)}|F_1(\psi_{y, \rho}) - F_1(u_\infty(\cdot-y))|=\int_{B_\rho(0)}|F'_1(\theta_{y,\rho})||\phi(x)-1|| u_\infty(\cdot-y))|,
		\end{equation}
		where $|\theta_{y,\rho}|\leq |\psi_{y,\rho}|+|u_\infty(\cdot-y)|$. Then, since $(\theta_{y,\rho}) \subset \mathbb{R}$ is a bounded and $F_1\in C^1(\mathbb{R})$, we derive that
		$$\int_{B_\rho(0)}|F'_1(\theta_{y,\rho})||\phi(x)-1||u_\infty(\cdot-y))|\leq C\int_{B_\rho(0)}|\phi(x)-1||u_\infty(0)|=o_\rho(1).$$
		From (\ref{038})-(\ref{037}),
		$$\int_{\mathbb{R}^N}F_1(\psi_{y,\rho}) \longrightarrow \int_{\mathbb{R}^N}F_1(u_\infty(\cdot-y)),\,\,\,\forall y\in \mathbb{R}^N.$$
		Adapting the ideas used in the proof of Lemma \ref{040}, we can show that $t_{y,\rho} \rightarrow 1$ as $\rho \rightarrow 0$, and so, 
		$$ ||\phi_\rho(y)||_{H^1(\mathbb{R}^N)}= ||t_{y,\rho}\psi_{y,\rho}||_{H^1(\mathbb{R}^N)}\longrightarrow|| u_\infty||_{H^1(\mathbb{R}^N)} \,\,\,\text{as}\,\,\,\rho \rightarrow 0,$$
		and
		$$\int_{\mathbb{R}^N}F_i(\phi_\rho(y))\longrightarrow \int_{\mathbb{R}^N}F_i(u_\infty),\,\,\, i\in \{1,2\}.$$
		The last convergences yield that
		$$\lim_{\rho \rightarrow 0} I_\infty (\phi_\rho(y)) \longrightarrow I_\infty(u_\infty)=d_\infty,$$
		uniformly in $y \in \mathbb{R}^N$, proving the part $i)$ of lemma.\\ 
		\textbf{Verification of $ii)$:} The proof follows as in the proof Lemma \ref{040} and it will be omitted.
	\end{proof}	
	
	A byproduct of the last lemma is the following corollary.
	\begin{corollary}\label{048}
		Given $\varepsilon \approx 0^+$, there exists $\rho_0>0$ such that
		$$
		\sup_{y \in \mathbb{R}^N} I_\infty(\phi_\rho(y))<2d_\infty-\varepsilon, \quad \forall \rho \in (0,\rho_0).
		$$	
	\end{corollary} 
	
	Next, we establish more two important properties of the mappings $\phi_\rho(y)$.
	\begin{lemma}\label{046}
		Fixed $\rho>0$, there exists $R_0>\rho$ such that
		\begin{enumerate}
			\item [$i)$:] $d_\infty< I(\phi_\rho(y))<\frac{c_0+d_\infty}{2},\,\,\,|y|\geq R_0$;
			\item [$ii)$:] $\langle \tau(\phi_\rho(y)), y \rangle,\,\,\, |y|=R_0$. 
		\end{enumerate}	
	\end{lemma}
	\begin{proof}
		\textbf{Verification of $i)$:} By the definition of $\phi_\rho(y)$, 
		$$d_\infty \leq I_\infty (\phi_\rho(y))=I(\phi_\rho(y)).$$
		On the other hand, as $d_\infty=d_0$ (see Lemma \ref{040}) and $(P_0)$ has no ground state solution, it follows that
		$$d_\infty < I(\phi_\rho(y)), \quad \mbox{for any} \quad \rho> 0 \quad \mbox{and} \quad y\in \mathbb{R}^N.$$
	
		Finally, note that, by part $ii)$ of Lemma \ref{041},
		$$I(\phi_\rho(y))<\frac{c_0+d_\infty}{2},\,\,\, |y|\geq R_0,$$
		for some $R_0>0$ large enough, because $c_0>d_\infty$. This completes the proof of item $i)$.\\
		\textbf{Verification of $ii)$:} The proof follows as in \cite[Lemma 4.3 (b)]{Benci-Cerami}.
	\end{proof}

	We finish this section by showing that $I|_\mathcal{N}$ satisfies the $(PS)_c$ for some levels $c \in \mathbb{R}$. 
	\begin{proposition}\label{050}
		For each fixed $\varepsilon \approx 0^+$, the functional $I|_\mathcal{N}$ satisfies the $(PS)_c$ condition for $c \in (d_\infty+\varepsilon,2d_\infty-\varepsilon)$. 
	\end{proposition}
	\begin{proof}
	Let $(u_n)$ be a $(PS)_c$ sequence for $I|_\mathcal{N}$. By Lemma \ref{limitasequencia}, we know that $(u_n)$ is a bounded sequence in $X$. Since $X$ is a reflexive space, we may assume that
		$$u_n \rightharpoonup u_0 \,\,\,\text{in}\,\,\, X.$$
		
		If $u_n\nrightarrow u_0$, by Lemma \ref{031} there are $u_1,...,u_k$ solutions of $(P_\infty)$ such that
		$$||u_n||_{H_0^1(\Omega)}^2 \longrightarrow ||u_0||_{H_0^1(\Omega)}^2+\sum_{j=1}^{k}||u_j||_{H^1(\mathbb{R}^N)}^2$$
		and
		$$ I(u_n) \longrightarrow I(u_0)+\sum_{j=1}^{k}I_\infty(u_j).$$
		Supposing that $u_0\neq 0$, we arrive at 
		$$I(u_n)\geq (k+1)d_\infty+o_n(1).$$
	Since $k \geq 1$, it follows that 
		$$c\geq (k+1)d_\infty\geq 2d_\infty,$$
	which is absurd, because $c<2d_\infty$. This contradiction allow us to infer that $u_0=0$. Moreover, we must have $k=1$,  because if $k>1$, then
		$$I(u_n) \geq kd_\infty \geq 2d_\infty,$$
	obtaining again a contradiction. From this, the unique possibility is $u_0=0$ and $u_1>0$, and so,  
		$$c+o_n(1) = I(u_n)=I_\infty(u_1)+o_n(1)= d_\infty + o_n(1).$$
		The last equality implies that $c=d_\infty$, which is absurd. This reasoning shows that $u_n \rightarrow u_0$ and the proof is finished.
	\end{proof}	
	
	\section{Existence of positive solution for $(P_0)$}	
	Along this section we show how the technical results of the preceding section imply in the existence of positive solution for $(P_0)$. The key point is to show that the functional $I$ possesses  a $(PS)_c$ sequence in a suitable level $c\in (d_\infty+\varepsilon,2d_\infty-\varepsilon)$, $\varepsilon \approx 0^+$. Bearing this in mind, set
	$$G:=\{\phi_\rho(y);\,|y|\leq R_0\}$$
	and
	$$ H:=\left\{\eta \in C(\mathcal{N}\cap P, \mathcal{N}\cap P);\,\,\eta(u)=u,\,\,\,\text{if}\,\,\,I(u)<\frac{c_0+d_\infty}{2}\right\}.$$
	Hereafter, we are using the same notations introduced in Section 4. Now, fix
	$$\Gamma:=\{\eta(G);\,\eta \in H\}$$
	and 
	$$c:= \inf_{A \in \Gamma}\sup_{u\in A} I(u).$$
	In view of Lemma \ref{046}-$ii)$, as made in \cite{Alves- de Freitas, Benci-Cerami}, we can prove the lemma below.	
	\begin{lemma}
		It holds 
		$$A\cap T_0 \neq \emptyset,\,\,\,\,\forall A \in \Gamma.$$	
	\end{lemma}
	
	Our second result in this section ensures that, for some convenient $\varepsilon>0$,  we must have \linebreak $c\in (d_\infty+\varepsilon,2d_\infty-\varepsilon)$, which is a key step to show the $(PS)_c$ condition of $I$ restricted to $\mathcal{N}$.
	\begin{lemma}
		There exists $\varepsilon >0$ such that $c\in (d_\infty+\varepsilon,2d_\infty-\varepsilon).$
	\end{lemma}	
	\begin{proof}
		Using the preceding lemma, for each $A \in \Gamma$ there exists $u_0 \in A\cap T_0$. Therefore,
		$$c_0 = \inf_{u \in T_0} I(u)\leq I(u_0) \leq \sup_{u\in A} I(u),$$
		and so,
		$$
		c_0\leq c.
		$$
		Take $\varepsilon \in (0, \frac{d_\infty}{2})$, $\varepsilon \approx 0^+$, such that
		\begin{equation}\label{049}
		d_\infty+\varepsilon < c_0 \leq c,
		\end{equation}
		which is possible in view of Lemma \ref{047}. On the other hand, since 
		$$c\leq \displaystyle{\sup_{u\in A}} I(u),\,\,\,\forall A \in \Gamma,$$
		we know that, 
		$$c\leq \sup_{\phi_\rho(y) \in G} I(\eta(\phi_\rho(y))),\,\,\,\forall \eta \in H.$$
		Choosing $\eta:=Id_{(\mathcal{N}\cap P)}$ and applying the Corollary \ref{048}, one finds 
				$$c < 2d_\infty-\varepsilon,$$
	for $\varepsilon$ and $\rho$ small enough. This combines with (\ref{049}) to give 
		$$c\in (d_\infty+\varepsilon,2d_\infty-\varepsilon).$$
	\end{proof}	
	Now we are able to prove that the problem $(P_0)$ has a positive solution.
	\\
	\\
	\textbf{Proof of Theorem \ref{053}:} Combining the preceding lemma with the Proposition \ref{050}, it suffices to show that $I|_{\mathcal{N}}$ has a $(PS)_c$ sequence in $P$. More precisely, we will prove that the following condition holds:\\
	\\
	\textbf{(D)}: For each $\lambda \in (0, c-\frac{c_0+d_\infty}{2})$, there exists $u_{\lambda} \in I^{-1}([c-\lambda, c+\lambda])$ with $u_\lambda \in \mathcal{N}\cap P$ and
	$$||I'(u_\lambda)||_*<\lambda.$$
	
	Arguing by contradiction, we find $\lambda_0\in (0, c-\frac{c_0+d_\infty}{2})$ such that
	$$||I'(u_\lambda)||_*\geq\frac{\lambda_0}{2},\,\,\,\forall u \in I([c-\lambda_0, c+\lambda_0])\cap (\mathcal{N}\cap P).$$
	By applying the version of quantitative deformation lemma in \cite{Willem}, we get \linebreak $\eta \in C([0,1]\times \mathcal{N}\cap P,\, \mathcal{N}\cap P)$ satisfying\\
	$i):$ $\eta(t,u) = u,\,\,\, \forall u \in I^{-1}([c-\lambda_0,c+\lambda_0]);$\\
	$ii):$ $\eta(1,I^{c+\frac{\lambda_0}{2}}) \subset I^{c-\frac{\lambda_0}{2}},$
	with $I^{d}:=\{u \in \mathcal{N}\cap P;\, I(u)\leq d\}.$
	
	By the definition of $c$, it holds
	$$\sup_{u\in A_0} I(u)\leq c+\frac{\lambda_0}{2},$$
	for some $A_0 \in \Gamma$, that is,
	$$ A_0 \in I^{c+\frac{\lambda_0}{2}}.$$
	Then, by item $ii)$,
	\begin{equation}\label{051}
	\eta(1, A_0) \in I^{c-\frac{\lambda_0}{2}}.
	\end{equation}
	Note that $A_0=\eta_0(G)$ for some $\eta_0 \in H$. Setting $\gamma_1:=\eta(1,\cdot)\circ \eta_0$ we derive that $\gamma_1 \in C(\mathcal{N}\cap P, \mathcal{N}\cap P)$ and, if $I(u) <\frac{c_0+d_\infty}{2}$, 
	$$\gamma_1(u)=\eta(1,\eta_0(u))=u$$
	(Note that $c-\lambda_0>\frac{c_0+d_\infty}{2}$). Thus, $\gamma_1 \in H$ and  
	$$  \eta_(1,A_0)=\eta(1,\eta_0(G))=\gamma_1(G)\in \Gamma.$$
	Consequently, by (\ref{051}),
	$$c\leq \sup_{u\in \eta(1, A_0)} I(u) \leq c-\lambda_0.$$
	This contradiction completes the proof. \hspace{8 cm} $\square$   
	\section{The Neumman case}
In this section, we study the existence of solution for the following class of problems
	$$\left\{\begin{aligned}
	-&\Delta u + u= Q(x)u\log u^2,\,\,\,\text{in}\,\,\,\Omega\\
	&\frac{\partial u}{\partial \eta}=0,\,\,\,\text{in}\,\,\,\partial \Omega,
	\end{aligned}
	\right.
	\leqno{(S_0)}
	$$
where $\Omega$ is an exterior domain  as in the problem $(P_0)$, and $Q:\mathbb{R}^N\longrightarrow \mathbb{R}$ is a continuous function satisfying the following conditions:\\ 
	\begin{itemize}
	\item [$(Q_1)$] $\displaystyle{\lim_{|x|\rightarrow \infty}} Q(x) = Q_0$ and  $q_0:=\displaystyle{\inf_{x\in \mathbb{R}^N}}\,Q(x)>0$ for all $x\in \mathbb{R}^N$;\\
	\item [$(Q_2)$] $Q_0\geq Q(x)\geq Q_0-Ce^{-M|x|^2}$, for $x\geq R_0$, $M\geq M_0$,\\
	\end{itemize}
	with $Q_0, C, M_0$, $R_0>0$.
	
The reader will see in this section that different of the Dirichlet case, we will prove that if $M_0>0$ is large enough, then the Problem $(S_0)$ has a ground state solution.

Let $(E,||\cdot||_E)$ be a Banach space and $d\in\mathbb{R}$. We recall that a  \textit{Cerami sequence} for a functional $J\in C^1(E,\mathbb{R})$ at level $d$ (shortly $(C)_d$-\textit{sequence}), is a sequence $(u_n) \subset E$ satisfying
	$$J(u_n)\longrightarrow d\,\,\, \text{and}\,\,\, (1+||u_n||_E)||J'(u_n)||_{E'}\longrightarrow 0.$$
	We say that $J$ verifies the Cerami condition at level $d$, or $(C)_d$-condition for short, if each $(C)_d$-sequence for $J$ admits a convergent subsequence. Note that a $(C)_d$-sequence for $J$ is also a $(PS)_d$-sequence. Therefore, if $u_n\rightarrow u_0$ and $(u_n)$ is a $(C)_d$-sequence, then $u_0$ is a critical point of $J$. See \cite{Cerami} for further details.

Hereafter, we will need of the auxiliary problem below
	$$\left\{\begin{aligned}
	-&\Delta u + u= Q_0u\log u^2,\,\,\,\text{in}\,\,\,\mathbb{R}^N\\
	&u\in H^1(\mathbb{R}^N).
	\end{aligned}
	\right.
	\leqno{(S_\infty)}
	$$
	Note that, in view of the condition $(Q_1)$, the problem $(S_\infty)$ is the limit problem of $(S_0)$.

	 Applying the Theorem \ref{1}, by a change of variable, we get the uniqueness of positive solution for $(S_\infty)$. In fact, if $u_1$ is a solution for \eqref{problemalimite}, by defining $v_1(x):= u_1(\sqrt{k^{-1}}x)$, by a direct computation, we find 
	 	$$-\Delta v_1=-v_1+\frac{1}{k}v_1\log v_1^2\,\,\,\text{in}\,\,\,\mathbb{R}^N.$$
	 	So, we get the existence and uniqueness of positive solution for $(S_\infty)$ by choosing $k=Q_0^{-1}$. 
    
 		From  now on, we may assume that, up to translations, the problem $(S_\infty)$ has a unique positive solution of the form
 	    \begin{equation}\label{055}
     		v_\infty(x)=C_1e^{-C_2|x|^2}, \quad \forall x \in \mathbb{R}^N,
        \end{equation}
	 for convenient $C_1$, $C_2>0$.
	 
	 Related with the problems $(S_0)$ and $(S_\infty)$ we have the energy functionals
	 $$J(u):=\frac{1}{2}\int_\Omega(|\nabla u|^2+(1+Q(x))|u|^2)+ \int_\Omega Q(x)F_1(u)-\int_\Omega Q(x)F_2(u),\,\,\,\, \forall u \in Z,$$
	 and
	 $$J_\infty(u):=\frac{1}{2}\int_{\mathbb{R}^N}(|\nabla u|^2+(1+Q_0)|u|^2)+ \int_{\mathbb{R}^N} Q_0F_1(u)-\int_{\mathbb{R}^N}Q_0 F_2(u),\,\,\,\, \forall u \in Y,$$
	 with $Z:=(H^1(\Omega)\cap L^{F_1}(\Omega),||\cdot||_Z)$, $||\cdot||_Z:=||\cdot||_{H^1(\Omega)}+||\cdot||_{L^{F_1}(\Omega)}$, and $Y$ is chosen as in the previous sections. Thus, $J\in C^1(Z,\mathbb{R})$, $J_\infty \in C^1(Y,\mathbb{R})$ and critical points of $J$ and $J_\infty$ correspond respectively to solutions of $(S)$ and $(S_\infty)$. 
	 
	 The Nehari sets associated with the functionals $J$ and $J_\infty$ respectively are given by
	 $$\mathcal{M}:=\{u \in Z-\{0\};\,J'(u)u=0\}$$
	 and
	 $$\mathcal{M}_\infty:=\{u \in Y-\{0\};\,J_\infty'(u)u=0\}.$$
	 Arguing as in the proof of Proposition \ref{07}, we also derive that the sets $\mathcal{M}$ and $\mathcal{M}_\infty$ are $C^1$-manifolds. Indeed, it suffices to replace $\Psi_0$ and $\Psi_\infty$ in the proof of Proposition \ref{07} by 
	 $$\tilde{\Psi}_0(u)=J(u)-\frac{1}{2}\int_{\Omega}Q(x)|u|^2\,\,\,\text{and}\,\,\, \tilde{\Psi}_\infty(u)=J_\infty(u)-\frac{1}{2}\int_{\mathbb{R}^N}Q_0|u|^2,$$
	 respectively. From now on, we will denote by $l_0$ and $l_\infty$ the levels
	 $$l_0:=\inf_{u\in \mathcal{M}} J(u)\,\,\,\text{and}\,\,\,l_\infty:=\inf_{u \in \mathcal{M}_\infty} J_\infty(u).$$
	 It is not difficulty to prove that the function $v_\infty$ given in (\ref{055}) satisfies 
	 \begin{equation}\label{056}
	 J_\infty(v_\infty)=l_\infty.
	 \end{equation}
	 The next result is a version of Lemma \ref{031} for the $(C)_d$-sequences of the functional $J$.
	 
	\begin{lemma}\label{057}\label{067}
	 	Let $(v_n)$ be a $(C)_d$-sequence for $J$. Assume that $v_n \rightharpoonup v_0$. Then, going to a subsequence if necessary, either
	 	\begin{description}
	 		\item [$i)$] $v_n \rightarrow v_0$ in $Z$, or
	 		\item [$ii)$] There exists $k\in\mathbb{N}$ and $k$ nontrivial solutions $v_j$  of $(S_\infty)$, $j\in \{1,...,k\}$, satisfying
	 		\begin{equation*}
	 		\left\|v_n - v_0-\sum_{j=1}^{k}v_n^j\right\|_{H^1(\Omega)}^2=o_n(1)\,\,\,\,\text{and}\,\,\,\,J(u_n)\rightarrow J(v_0)+\sum_{j=1}^{k} J_\infty(u_j),
	 		\end{equation*}
	 	\end{description}
 	with $v_n^j:=v_j(\cdot\,\,-y_n^j)$, and $(y_n^j) \subset \mathbb{R}^N$ with $|y_n^j|\rightarrow \infty$ for each $j\in \{1,...,k\}$.	
	 \end{lemma}
 	\begin{proof}
 	The proof is a slight variant of the argument made in Lemma \ref{031} (see also the ideas in \cite[Lemma 3.3]{Alves-Carriao-Medeiros} and \cite[Lemma 3.1]{Benci-Cerami}). In fact, since $(v_n)$ is $(C)_d$-sequence for $J$, it holds $J'(v_n)v_n=o_n(1)$. So, it is possible to prove that $(v_n)$ is bounded in the same way of the proof of Lemma \ref{031}. From this, it follows that $(v_n)$ is a bounded $(PS)_d$ sequence for $J$. Accounting that $v_n \rightharpoonup v_0$, we derive that $J'(v_0)=0$, and so, $v_0$ is a solution of $(S_0)$. Following the ideas in the proof of Lemma \ref{031}, setting
 	$$
 	\xi_n^1(x):=
 	v_n(x)-v_0(x),\,\,\,\text{in}\,\,\,\Omega;
 	$$
 	we find that 
 	$$\xi_n^1 \rightharpoonup 0\,\,\,\text{in}\,\,\, Z.$$ 
 	Then, if $\xi_n^1 \rightarrow 0$ in $Z$, the proof would be finished. Otherwise, if $\xi_n^1 \not\rightarrow 0$ in $Z$, arguing as in the proof of Lemma \ref{031}, see items $i)-ii)$, we find
 	\begin{equation}\label{069}
 	J(\xi_n^1)=J(v_n)-J(v_0)+o_n(1)
 	\end{equation}
 	and
 	\begin{equation}\label{072}
 	J'(\xi_n^1)\xi_n^1=J'(v_n)v_n-J'(v_0)v_0+o_n(1).
 	\end{equation}
 	
 	In the same line of Lemma \ref{031}, let us consider $(y_n^1)_{n \in \mathbb{N}}$ in $\mathbb{R}^N$, with $y_n^1$ the centers of unit $N$-dimensional hypercubes $B_i$, $\mathbb{R}^N=\displaystyle{\bigcup_{i\in \mathbb{N}}}B_i$, and verify 
 	$$||\xi_n^1||_{L^p(\tilde{B_i})}^p=\max_{j\in \mathbb{N}} ||\xi_n^1||_{L^p(\tilde{B_j})}^p:=\delta_n,$$ 
 	where $\tilde{B}_i=(B_i\cap\Omega)$. Next, we are going to guarantee that
 	$$\delta_n\geq\tau_0 >0,\,\,\, n\geq n_0,$$
 	for some $n_0 \in \mathbb{N}$, and 
 	$$|y_n^1|\rightarrow \infty.$$
 	In the sequel, we set
 	$$\tilde{\xi}_n(x)=\xi_n^1(x+y_n^1), \quad \Omega_n^1=\{y-y_n^1;\, y\in \Omega\},\,\,\,X_n:=H^1(\Omega_n^1)\cap L^{F_1}(\Omega_n^1)$$
 	and the functional $J_n:X_n\longrightarrow \mathbb{R}$ given by
 	$$J_n(u):=\frac{1}{2}\int_{\Omega_n^1}(|\nabla u|^2+(1+Q(x+y_n^1))|u|^2)+ \int_{\Omega_n^1} Q(x+y_n^1)F_1(u)-\int_{\Omega_n^1} Q(x+y_n^1)F_2(u), \quad u\in X_n.$$
 	The following claim holds.
 	\begin{claim}
 		The sequence $\tilde{\xi}_n$ is such that
 		\begin{equation}\label{071}
 		J_n(\tilde{\xi_n})\geq \tau_1>0,
 		\end{equation}
 		for some $\tau_1\in \mathbb{R}$.
 	\end{claim} 
 	It suffices to show that
 	$$\inf_{n \in \mathbb{N}}\left(\frac{1}{2}\int_{\Omega_n^1}(|\nabla \tilde{\xi}_n|^2+(1+Q(x+y_n^1))|\tilde{\xi}_n|^2)+ \int_{\Omega_n^1} Q(x+y_n^1)F_1(\tilde{\xi}_n)-\int_{\Omega_n^1} Q(x+y_n^1)F_2(\tilde{\xi}_n)\right)>0.$$
 	Arguing as in the Claim \ref{029}, by considering (\ref{072}) and the condition $(Q_1)$, we find
 	$$ J_n(\tilde{\xi}_n)=\int_{\Omega_n^1}Q(x+y_n^1)|\tilde{\xi}_n|^2+o_n(1)\geq q_0\int_{\Omega_n^1}|\tilde{\xi}_n|^2+o_n(1).$$
 	Now, if $J_n(\tilde{\xi}_n)=o_n(1)$, then it would have $||(\chi_{\Omega_n^1}\tilde{\xi}_n)||_{L^2(\mathbb{R}^N)}^2=o_n(1)$, and so $\displaystyle{\int_{\mathbb{R}^N}}|\chi_{\Omega_n^1}\tilde{\xi}_n|^p=o_n(1)$, for a fixed $p\in (2,2^*]$, by an interpolation argument. From this, by the properties on $F_2$ (\textit{vide} $(P_2)$ above), it follows that
 	$$\int_{\Omega_n^1}F'_2(\tilde{\xi}_n)\tilde{\xi}_n=\int_{\mathbb{R}^N}F'_2(\chi_{\Omega_n^1}\tilde{\xi}_n)\chi_{\Omega_n^1}\tilde{\xi}_n=o_n(1).$$
 	Therefore,
 	$$\int_{\Omega_n^1}(|\nabla \tilde{\xi}_n|^2+(1+Q(x+y_n^1))|\tilde{\xi}_n|^2)+ \int_{\Omega_n^1} Q(x+y_n^1)F'_1(\tilde{\xi}_n)\tilde{\xi}_n=o_n(1).$$
 	Equivalently, by a change of variable,
 	$$\int_{\Omega}(|\nabla \xi_n|^2+(1+Q(x))|\xi_n|^2)+ \int_{\Omega} Q(x)F'_1(\xi_n)\xi_n=o_n(1),$$
 	contradicting the fact that $\xi_n \nrightarrow 0$. The proof of the claim is completed.
 	
 	In the same line of Lemma \ref{031}, we are able to show that the next claim holds.
 	
 	\begin{claim} There exist $\tau_0>0$ and $n_0 \in \mathbb{N}$ such that
 		$$\delta_n \geq \tau_0,\,\,\,n\geq n_0.$$
 	\end{claim}
 	
 	Take into accounting the inequality in (\ref{071}), the proof of the claim follows by reasoning as made in Claim \ref{029}. However, we would like point out an important fact related with the proof of the Claim \ref{029}. The inequality in (\ref{070}) plays a crucial role in the proof of Claim \ref{029}. Such inequality is based in the fact that the constant associated with the embedding
 	$$H^1(B_i) \hookrightarrow L^p(B_i)$$
 	are independent of $i$. In the current proof a similar property also holds, more precisely
 	$$H^1(\tilde{B}_i) \hookrightarrow L^p(\tilde{B}_i),$$
 	since the sets $\tilde{B}_i=(B_i\cap \Omega)$ verify the uniform cone property (see \cite{Adams1}). 
	
	The preceding claim assures that
	$$|y_n^1|\longrightarrow \infty.$$ 
	In fact, otherwise, it would be possible to find $R>0$, such that
	$$\int_{(B_R(0)\cap \Omega)}|\xi_n^1|^p\geq \int_{\tilde{B}_i}|\xi_n^1|^p= \delta_n^p\geq \tau_0^p>0.$$
	This contradicts the convergence
	$$\int_{(B_R(0)\cap\Omega)} |\xi_n^1|^p \longrightarrow 0,$$
	which is valid in view of the weak convergence $\xi_n^1 \rightharpoonup 0$ in $Z$. Thus, hereafter we will assume that $|y_n^1|\rightarrow \infty.$
	
 	Now, since $y_n^1\rightarrow \infty$, we know that $\Omega_n^1 \rightarrow \mathbb{R}^N$, as $n\rightarrow \infty$, (in the sense of the characteristic functions $\chi_{\Omega_n^1}\rightarrow 1$ a.e. in $\mathbb{R}^N$) for each $R>0$, there exists $m_0\in \mathbb{N}$ such that $B_R(0) \subset \Omega_n^1$, $n\geq m_0$. Considering that $(\xi_n^1)$ is a bounded sequence, it is possible to find $v_1\in Y\setminus \{0\}$ satisfying
 	$$\tilde{\xi_n} \rightharpoonup v_1 \,\,\,\text{in}\,\,\, H^1(B_R(0))\cap L^{F_1}(B_R(0)),$$
 	for each $R>0$ fixed. Fixed $\phi \in C_0^\infty(\Omega)$, inasmuch as $|y_n^1|\rightarrow \infty$, we know that, for some $m_1\in \mathbb{N}$, it holds 
 	$$\text{supp}\,\phi(\cdot-y_n^1)\subset \Omega,\,\,\,n\geq m_1.$$ 
 	Hence,  $\phi^{(y_n^1)}:=\phi(\cdot-y_n^1)\in C_0^{\infty}(\Omega)$ for $n\geq m_1$.  
 	
 	By exploring the ideas in the proof of Lemma \ref{031}-$ii)$, we derive
  	$$\sup_{n\in \mathbb{N}}\,\left(|J'(\xi_n)|\cdot||\phi(\cdot\,-y_n^1)||_Z\right)=o_n(1).$$

 	By combining these information with the properties $(Q_1)$ and (\ref{070}) above, we derive that $v_1$ is a nontrivial solution of $(S_\infty)$. Set
 	$$\xi_n^2:=\xi_n^1-v_1(\cdot-y_n^1),\,\,\,\text{in}\,\,\,\Omega.$$
 	Hence, we can repeat the preceding steps made with $\xi_n^1$. Following this procedure, the reasoning made in final of Lemma \ref{031} allow us to get a $k\in \mathbb{N}$ and unbounded sequences $(y_n^1),...,(y_n^k)$  in $\mathbb{R}^N$ in such way that
 	$$\xi_n^j:=\xi_n^{j-1}(\cdot+y_n^{j-1})-v_{j-1}\rightharpoonup 0,\,\,\,\text{in}\,\,\, Y,$$
 	with $v_{j-1}$ a nontrivial solution of $(S_\infty)$, $\xi_n^{k+1} \rightarrow 0$, as $n\rightarrow \infty$, $j\in \{2,...,k\}$. Setting $v_j:=v_j(\cdot-y_n^j)$, these facts assure that
 	$$\left\|v_n - v_0-\sum_{j=1}^{k}v_n^j\right\|_{H^1(\Omega)}^2=o_n(1)$$
 	as well as
 	$$J(u_n)\longrightarrow J(v_0)+\sum_{j=1}^{k} J_\infty(u_j).$$
 	\end{proof}
 	An immediate consequence of the preceding lemma is following corollary.
 	\begin{corollary}\label{064}
 		The functional $J$ verifies the $(C)_d$-condition for $d\in(0,l_\infty)$.
 	\end{corollary}
 	\begin{proof}
 		Let $(v_n)$ be a $(C)_d$-sequence, with $d\in (0,l_\infty)$. In particular,
 		$$J'(v_n)v_n = o_n(1),$$
 	and	so, using the same ideas explored in the begin of the proof of Lemma \ref{050}, we derive that $(v_n)$ is a bounded sequence in $Z$ and, going to a subsequence if necessary, it holds $v_n \rightharpoonup v_0$, for some $v_0\in Z$. Since $(v_n)$ is a $(C)_d$-sequence, we have $J'(v_0)=0$. Now, it is sufficient to observe that the hypothesis $d\in (0,l_\infty)$ combined with the items $i)-ii)$ of the preceding lemma gives the required result.
 	\end{proof}
 	
 	We are going to show that $J$ has a ground state solution, i.e., a positive solution $v_0$ satisfying $J(v_0)=l_0$. We start by showing that the functional $J$ satisfies the mountain geometric (see e.g \cite[Section 2.3]{Willem}).
 	
 	\begin{lemma}
 		The functional $J$ verifies the Mountain Pass geometry, i.e.,
 		\begin{itemize}
 		\item [$i)$] $J(0)=0$ and there exist $r$, $\rho_0>0$ such that $J_{\partial B_r(0)}\geq \rho_0$;
 		\item [$ii)$] There exits $v$, $||v||_Z> r$, and $J(v)\leq J(0)=0$.  	
 		\end{itemize}
 	\end{lemma}
 	\begin{proof}
 	$i)$: From the conditions $(Q_1)-(Q_2)$ it follows that, for some constant $C>0$, it holds
 	$$J(u)\geq C||u||^2_{H^1(\Omega)}+C\int_{\Omega}F_1(u)-Q_0\int_{\Omega}F_2(u).
 	$$
 	Using (\ref{2}) and $(P_2)$, modifying the constant $C$ if necessary, we can find $r\approx0^+$ such that, for $||u||_Z=r$, is valid that
 	$$
 	J(u)\geq C||u||^2_{H^1(\Omega)}+C||u||_{L^{F_1}(\Omega)}^2-C_1||u||_Z^p \geq C_2||u||_Z^2-C_1||u||_Z^p
 	$$
 	with $C_1$, $C_2>0$ and $p>2$. The property required in the item $i)$ follows as a direct consequence of the last inequality.\\
 	$ii)$: Fix $u\in Z-\{0\}$. So,
 	$$J(tu)=\frac{t^2}{2}\left[\int_{\Omega}(|\nabla u|^2+|u|^2)-\frac{1}{2}\int_{\Omega}Q(x)u^2\log u^2-\log t\int_{\Omega}Q(x)u^2\right]\longrightarrow -\infty,$$
 	as $t\rightarrow \infty$. So, the item $ii)$ holds by taking $v=tu$, for some $t\approx \infty$.
	\end{proof} 

	We are going to show that the problem $(S_0)$ has a ground state solution. To begin with, we will show the existence of a $(C)_{d}$-sequence at mountain pass level. Namely, we have the following corollary. 
	
	\begin{corollary}\label{063}
		The functional $J$ has a sequence $(C)_{\tilde{l}_0}$-sequence, where $\tilde{l}_0$ is the level
		$$\tilde{l}_0:=\inf_{\gamma \in \Gamma} \sup_{t\in [0,1]} J(\gamma(t)),$$
		and
		$$\Gamma:=\{\gamma \in C([0,1],Z);\,\gamma(0)=0,\,\,\gamma(1)<0\}.$$ 
	\end{corollary}
	\begin{proof}
	The result follows by a variant of the classical Mountain Pass Theorem of Ambrosetti-Rabinowitz (see, e.g., \cite[Section 2]{Willem}). Note that the reasoning made in \cite{Willem} can be adapted when the $(PS)_d$-sequences are replaced by $(C)_d$-sequences (see the Proposition 1.1 in \cite{Cheng-Wu-Liu} for a statement of a variant Mountain Pass Theorem involving the Cerami sequences).
	\end{proof}
 
 Exploring the ideas in \cite[Lemma 3.3]{Alves-de Morais}, in view of $(Q_1)$, we can show that the level $\tilde{l}_0$ in the above corollary coincides with the level $l_0$, namely, it holds
 	\begin{equation}\label{058}
 	\tilde{l}_0=l_0:=\inf_{u\in \mathcal{M}} J(u).
 	\end{equation}
 	Thereby, the last corollary assures the existence of a $(C)_{l_0}$-sequence for $J$. The next lemma is our main technical result in the present section, and it relates the levels $l_0$ and $l_\infty$.
 	
 	\begin{lemma}\label{065}
 		Assume the conditions $(Q_1)-(Q_2)$. Then the following inequality holds.
 		$$l_0<l_\infty.$$
 	\end{lemma}	
	\begin{proof}
		Set
		$$v_n(x):=v_\infty(x-x_n),$$
		with $x_n:=(n,0,...,0) \in \mathbb{R}^N$ and $v_\infty$ the solution of $(S_\infty)$ satisfying (\ref{056}). By (\ref{058}), 
		$$l_0\leq \max_{t\geq 0}\,J(tv_n)=:J(t_nv_n),$$
		and $t_n\in (0,\infty)$. In this way, we derive that $t_nv_n \in \mathcal{M}$, which yields
		$$t_n^2\int_{\Omega}(|\nabla v_n|^2+|v_n|^2)=\int_{\Omega}t_n^2|v_n|^2\log |t_nv_n|^2.$$
		Therefore, since $|x_n|\rightarrow \infty$, the same ideas employed in the proof of Lemma \ref{040} enable us to show that, going to a subsequence if necessary, it holds $t_n\rightarrow 1$. 
		
		Now, it follows that
		$$\begin{aligned}
		l_0\leq J(t_nv_n)&=\frac{1}{2}\int_{\Omega}(|t_n\nabla v_n|^2+(1+Q(x))|t_nv_n|^2)+\int_{\Omega}Q(x)F_1(t_nv_n)-\int_{\Omega}Q(x)F_2(t_nv_n)=\\
		&=J_\infty(t_nv_n)-\frac{t_n^2}{2}A_n+\int_{\Omega^c}Q_0F_2(t_nv_n)-\int_{\Omega^c}Q_0\left[F_1(t_nv_n)+\frac{t_n^2}{2}v_n^2\right]+\\
		&\,\,\,+\int_{\Omega}(Q_0-Q(x))\left[F_2(t_nv_n)-F_1(t_nv_n)-\frac{t_n^2}{2}v_n^2\right],
		\end{aligned}
		$$
		with $A_n:=\displaystyle{\int_{\Omega^c}}(|\nabla v_n|^2+|v_n|^2)$. From $(Q_1)$, 
		$$l_0\leq J_\infty(t_nv_n)-\frac{t_n^2}{2}A_n+\int_{\Omega^c}Q_0F_2(t_nv_n)+\int_{\Omega}(Q_0-Q(x))F_2(t_nv_n).$$
		Taking into account that $t_n\rightarrow 1$ as $|x_n|\rightarrow \infty$, the condition $(Q_1)$ and the invariance by translations of $\mathbb{R}^N$, one finds
		$$J_\infty(t_nv_n)=J_\infty(v_\infty)+o_n(1)=c_\infty+o_n(1).$$ 
		This information together with the last inequality give
		\begin{equation}\label{059}
		l_0\leq l_\infty+o_n(1)-\frac{t_n^2}{2}A_n+B_n,
		\end{equation}
		with $B_n:=\displaystyle{\int_{\Omega^c}}Q_0F_2(t_nv_n)+\displaystyle{\int_{\Omega}}(Q_0-Q(x))F_2(t_nv_n)$.
		
		Our next step is proving that $\dfrac{B_n}{A_n}\rightarrow 0$. Having this in mind, since $|\Omega^c|< \infty$, the equality in (\ref{055}) implies
		\begin{equation}\label{060}
			A_n\geq \int_{\Omega^c} |v_n|^2\geq Ce^{-2C_2n^2}, \quad \forall n \in \mathbb{N},
		\end{equation}
		for a convenient $C>0$. 
		From the condition $(P_2)$, for some $p\in (2,2^*]$, it holds 
		$$|F_2(t)|\leq C_p|t|^p,\,\,\,\forall t\in \mathbb{R}.$$
		Therefore, using again $|\Omega^c|<\infty$, one has 
		\begin{equation}\label{066}
		Q_0\int_{\Omega^c}F_2(t_nv_n)\leq Ce^{-pC_2n^2},
		\end{equation}
		for some $C$. Now, take $R_n\in(0,n)$. So,
		$$\int_{\Omega}(Q_0-Q(x))F_2(t_nv_n)=\int_{\Omega\cap [|x|>R_n]}(Q_0-Q)F_2(t_nv_n)+\int_{\Omega\cap [|x|\leq R_n]}F_2(t_nv_n).$$
		By invoking the assumption $(Q_2)$, it follows that
		\begin{equation}\label{061}
			\int_{\Omega\cap [|x|>R_n]}(Q_0-Q(x))F_2(t_nv_n)\leq Ce^{-MR_n^2},
		\end{equation}
		for some $C>0$, as well as,
		\begin{equation}\label{062}
		\int_{\Omega\cap [|x|\leq R_n]}(Q_0-Q(x))F_2(t_nv_n)\leq C_Nn^Ne^{-pC_2(n-R_n)^2}.
		\end{equation}
		for some constant $C_N>0$. The estimates in (\ref{060})-(\ref{062}) combined produce, for some constant $C>0$,
		$$
		\frac{B_n}{A_n}\leq C\left(\frac{e^{2C_2n^2}}{e^{pC_2n^2}}+ \frac{e^{2C_2n^2}}{e^{MR_n^2}} + \frac{C_Nn^Ne^{2C_2n^2}}{e^{pC_2(n-R_n)^2}} \right).
		$$
		Setting $R_n:=\dfrac{n}{k}$, $k\in \mathbb{N}$, we find
		$$\frac{C_Nn^Ne^{2C_2n^2}}{e^{pC_2(n-R_n)^2}}=\frac{C_Nn^Ne^{2C_2n^2}}{e^{(\frac{k-1}{k})^2pC_2n^2}}.$$
		Since $\left(\dfrac{k-1}{k}\right)^2$ converges to $1$, as $k\rightarrow \infty$, and $p>2$, we may fix $k_0\approx \infty$ such that $p\left(\frac{k_0}{k_0-1}\right)^2>2$. Hence
		$$\frac{C_Nn^Ne^{2C_2n^2}}{e^{(\frac{k_0-1}{k_0})^2pC_2n^2}} \,\,\,\longrightarrow \,\,\,0.$$
		Then, choosing $M_0$ large enough in the condition $(Q_2)$, we derive that
		$$\frac{e^{2C_2n^2}}{e^{MR_n^2}}=\frac{e^{2C_2n^2}}{e^{(M/k_0^2)n^2}}\longrightarrow 0.$$
		These convergences assure that
		$$\frac{B_n}{A_n} \longrightarrow 0.$$
	Recalling that $t_n \rightarrow 1$ for some $n_0 \in \mathbb{N}$,
		$$ -\frac{t_n^2}{2}A_n+B_n=\left(\frac{-t_n^2}{2}+\frac{B_n}{A_n}\right)A_n<0,\,\,\,n\geq n_0.
		$$
		Using this information in (\ref{059}), we derive that
		$$l_0<l_\infty,$$
		proving the desired result.
	\end{proof}
	
	Now we can prove our main result.

	\begin{proof}[Proof of Theorem \ref{068}]
		The proof is essentially established. In fact, by combining the Corollary \ref{063} with  (\ref{058}), there exists a $(C)_{l_0}$-sequence for $J$, which will be denotes by $(v_n)$. Since $(v_n)$ is bounded, it follows that 
		$$J(v_n)\longrightarrow l_0\,\,\,\text{and}\,\,\,J'(v_n)\longrightarrow 0.$$
		Invoking together the Corollaries \ref{064} and \ref{065}, we may assume that
		$$v_n\longrightarrow v_0\,\,\,\text{in}\,\,\,Z,$$
		for some $v_0$. In this way, we derive that 
		$$J(v_0)=l_0\,\,\,\text{and}\,\,\,J'(v_0) = 0,$$
		and so, $v_0$ is a ground state solution for $(S_0)$. Now, we would like to point out that $v_0$ can be chosen as a positive solution. Indeed, writing $v_0=v_0^+ - v_0^-$, with $v_0^+:=\max\{v_0,0\}$ and $v_0^-:=\max\{-v_0,0\}$, we find $J'(v_0^+)v_0^+=J'(v_0^-)v_0^-=0$ and $l_0=J(v_0)=J(v_0^+)+J(v_0^-)$. These facts combined assure that either $v_0^+=0$ or $v_0^-=0$. Hence, since $f(t)=t\log t$ is an odd function, we may assume that $v_0\geq0$, so that $v_0>0$ by a variant of maximum principle presented in \cite{Vazquez} (see \cite{Alves-de Morais,Alves-Ji, Alves-S da Silva} for a similar reasoning)
		\end{proof}

	
\end{document}